\documentclass[11pt]{amsart}
\usepackage[utf8]{inputenc}
\usepackage{amsmath,amssymb,amsthm,colonequals,mathrsfs,mathtools}
\usepackage{array,enumitem,yfonts}
\usepackage{booktabs,longtable}
\usepackage{algorithm}
\usepackage{algpseudocode}
\usepackage[alphabetic]{amsrefs}
\usepackage[all,cmtip]{xy}
\usepackage[colorlinks,anchorcolor=blue,citecolor=blue,linkcolor=blue,urlcolor=blue,bookmarksdepth=2]{hyperref}
\allowdisplaybreaks

\usepackage{comment}

\usepackage[margin=1in]{geometry}

\usepackage{tikz}
\usetikzlibrary{cd,arrows}
\tikzset{>=stealth}
\tikzcdset{arrow style=tikz}
\tikzset{link/.style={column sep=1.8cm,row sep=0.16cm}}

\AtBeginDocument{%
	\def\MR#1{}
}
\usepackage{fancyhdr}
\usepackage{amsmath}
\usepackage{mathdots}
\usepackage{amssymb}
\usepackage{amsthm}
\usepackage{here}
\usepackage{amscd} 
\usepackage{mathrsfs}
\usepackage{mathtools}
\usepackage{scalefnt}
\usepackage{url}
\theoremstyle{plain}
\newtheorem{thm}{Theorem}[section]
\newtheorem{lem}[thm]{Lemma}
\newtheorem{cor}[thm]{Corollary}
\newtheorem{prop}[thm]{Proposition}

\theoremstyle{definition}

\newtheorem{rem}[thm]{Remark}

\newtheorem{defn}[thm]{Definition}

\newtheorem{ex}[thm]{Example}
\numberwithin{equation}{section}

\def\Q{{\mathbb Q}}

\def\id{\mathop{\mathrm{id}}\nolimits}

\def\Aut{\mathop{\mathrm{Aut}}\nolimits}

\def\Lie{\mathop{\mathrm{Lie}}\nolimits}
\def\Hom{\mathop{\mathrm{Hom}}\nolimits}

\def\GL{\mathop{\mathrm{GL}}\nolimits}
\def\SL{\mathop{\mathrm{SL}}\nolimits}
\def\SU{\mathop{\mathrm{SU}}\nolimits}

\def\Spec{\mathop{\mathrm{Spec}}\nolimits}

\def\det{\mathop{\mathrm{det}}\nolimits}
\def\dim{\mathop{\mathrm{dim}}\nolimits}
\def\div{\mathop{\mathrm{div}}\nolimits}

\def\Stab{\mathop{\mathrm{Stab}}\nolimits}

\def\Hom{\mathop{\mathrm{Hom}}\nolimits}
\def\diag{\mathop{\mathrm{diag}}\nolimits}

\def\ord{\mathrm{ord}}

\def\tor{\mathrm{tor}}

\def\L{\mathscr{L}}

\def\OO{\mathscr{O}}

\def\exp{\mathop{\mathrm{exp}}\nolimits}

\def\SO{\mathop{\mathrm{SO}}\nolimits}
\def\SU{\mathop{\mathrm{SU}}\nolimits}

\def\U{\mathrm{U}}
\def\O{\mathrm{O}}

\newcommand{\bbA}{\mathbb{A}}

\newcommand{\bbC}{\mathbb{C}}

\newcommand{\bbQ}{\mathbb{Q}}

\newcommand{\bbR}{\mathbb{R}}

\newcommand{\bbZ}{\mathbb{Z}}

\newcommand{\calD}{\mathcal{D}}

\newcommand{\calE}{\mathcal{E}}

\newcommand{\calF}{\mathcal{F}}

\newcommand{\calS}{\mathcal{S}}

\newcommand{\calV}{\mathcal{V}}

\newcommand{\frakH}{\mathfrak{H}}

\newcommand{\Mat}{\mathrm{Mat}}

\newcommand{\fini}{\mathrm{fin}}

\newcommand{\vep}{\varepsilon}

\newcommand{\abcd}{\begin{pmatrix}a & b \\ c & d \end{pmatrix}
}

\newcommand{\Irr}{\mathrm{Irr}}

\newcommand{\BC}{\mathrm{BC}}
\newcommand{\disc}{\mathrm{disc}}

\allowdisplaybreaks[4]

\newcommand{\defeq}{\vcentcolon=}

\usepackage{comment}

\begin{document}
\title[On the singularities and the Kodaira dimension of unitary Shimura varieties]{On the singularities and the Kodaira dimension of unitary Shimura varieties}
\author{Shuji Horinaga$^{1}$ \and Yota Maeda$^{2,3}$}
\email{syuuji.horinaga@ntt.com, shorinaga@gmail.com, y.maeda.math@gmail.com}

\def\l@subsection{\@tocline{2}{0pt}{2.3pc}{5pc}{}}

\maketitle
\vspace{-1em}
\begin{center}
  \begin{minipage}{0.9\textwidth}
    \centering
    {\small
    $^{1}$ NTT Institute for Fundamental Mathematics, NTT Corporation, Japan\\
    $^{2}$ Fachbereich Mathematik, Technische Universität Darmstadt, Germany.\\
    $^{3}$ Mathematical Institute, Tohoku University, Japan\\
    }
  \end{minipage}
\end{center}

\vspace{1em}

\begin{abstract}
We study the geometry of Shimura varieties associated with a Hermitian form of signature $(p,q)$ over an imaginary quadratic field $E$, where $2\leq p\leq q$.
We prove that when $p\geq w_E$ and $(p,q)\neq(2,2),(2,3)$, where $w_E:=\#\OO_E^\times$, there exists a toroidal compactification with at worst canonical singularities.
As an application, combining this singularity analysis with Arthur's multiplicity formula, we prove that only finitely many pairs $(p,q)$ satisfying the above conditions and $p+q\equiv1\pmod{w_E}$ give rise to unitary Shimura varieties that are not of general type.
Our method also improves the singularity bound of Gritsenko--Hulek--Sankaran (Invent.\ Math., 2007) for $\O^+(2,n)$ to the range $n\geq6$ and shows that this bound is sharp.
\end{abstract}

\section{Introduction}

The geometry and birational classification of arithmetic quotients of Hermitian symmetric domains
have been a central theme in the study of Shimura varieties.
Given a connected, non-compact Hermitian symmetric domain $\mathcal{D}$ and a neat arithmetic subgroup
$\Gamma \subset \mathrm{Aut}(\mathcal{D})$, the quotient $X := \Gamma \backslash \mathcal{D}$
is a quasi-projective variety which carries rich arithmetic and geometric structures.
A fundamental problem is to understand the birational type of $X$, and in particular
to determine when $X$ is of general type, in connection with Lang's conjecture.

A powerful general approach to this problem was initiated by Ash, Mumford, Rapoport, and Tai in their work
on toroidal compactifications of locally symmetric varieties \cite{ash2010smooth}.
Choosing a smooth admissible fan $\Sigma$, a toroidal compactification $\overline{X}^{\Sigma} \supset X$ has a simple normal crossing boundary divisor $\Delta := \overline{X}^{\Sigma} \setminus X$.
Building on the Hirzebruch proportionality principle, Mumford showed that
automorphic vector bundles admit canonical extensions to toroidal
compactifications and established a proportionality theorem for their
Chern numbers  \cite{mumford1977hirzebruch}.
In this neat setting, $\overline{X}^{\Sigma}$ is smooth
and $\Delta$ is a simple normal crossing divisor.
Moreover, the log canonical divisor
$K_{\overline{X}^{\Sigma}}+\Delta$ is big.

We now allow $\Gamma$ to be an arbitrary arithmetic
subgroup.
With suitable choices in the toroidal construction,
the resulting compactification is locally described
by finite quotients of smooth charts.
Nontrivial stabilizers can produce quotient singularities,
both in the interior and along the boundary,
and these singularities need not be canonical.
Thus, determining the birational type at arbitrary
arithmetic level requires both an analysis of these
quotient singularities and the construction of
pluricanonical forms that extend across the boundary
and to a resolution.

The aim of this paper is to develop both the analysis of  singularities and the birational classification for Shimura varieties associated with unitary groups of signature $(p,q)$.
More precisely, let $E$ be an imaginary quadratic field and $L$ a Hermitian $\OO_E$-lattice of signature $(p,q)$ with $p,q>1$.
Let $G$ be the corresponding unitary group scheme over $\bbZ$ and $\mathcal{D}_{p,q}$ the associated Hermitian symmetric domain.
For an arithmetic subgroup $\Gamma \subset G(\bbZ)$, we consider the
\emph{unitary modular variety} $\mathrm{Sh}^0(\Gamma) \defeq \Gamma\backslash \mathcal{D}_{p,q}$.
This admits a toroidal compactification $\overline{\mathrm{Sh}^0(\Gamma)}^{\Sigma}$ with boundary divisor $\Delta$.

Our first main result shows that, under explicit
conditions on $(p,q)$ and $E$, the toroidal
compactification
$\overline{\mathrm{Sh}^0(\Gamma)}^{\Sigma}$
can be chosen to have at worst canonical singularities;
see Theorem \ref{mainthm:singularity}.
Combining this singularity analysis with the
construction of low weight cusp forms via Arthur's
multiplicity formula, we obtain sufficient conditions
for $\mathrm{Sh}^0(\Gamma)$ to be of general type;
see Theorem~\ref{mainthm:general type}.

These results provide finer birational information on unitary Shimura varieties.
In particular, compactifications with canonical singularities are better suited to techniques from the minimal model program.
The methods developed here, combining a detailed analysis of quotient singularities with automorphic representation theory, may also be applicable to other classes of Shimura varieties.

\section{Main results}
\subsection{Shimura varieties}

Let $E$ be an imaginary quadratic field with discriminant $-D$, and $\OO_E$ be its ring of integers with the unit group $\OO_E^\times$ of order $w_E$.
Let $L$ be a Hermitian lattice over $\OO_E$ with signature $(p,q)$ and $G$ be the associated group scheme over $\bbZ$.
Unless otherwise stated, we assume throughout this paper $p,q > 1$ and $p \leq q$.
Put $N = p+q$.
Associated with these data, the type $I_{p,q}$ Hermitian symmetric domain is defined as
\begin{align}\label{def_realization_D_p,q}
  \mathcal{D}_{p,q}&\defeq\{W \subset L \otimes_{\OO_E} \bbC \mid \dim_{\bbC} W = q,\ h|_W\ \text{is negative definite} \}\\
  &\cong \{Z\in \Mat_{p, q}(\bbC)\mid I_q - Z^*Z>0\}.\notag
\end{align}
Note that if $p=1$, this space coincides with the usual $q$-dimensional complex ball.
The unitary group $G(\bbZ)=\U(L)$ acts on this space.
For an arithmetic subgroup $\Gamma\subset\U(L)$, we define the unitary modular variety $\mathrm{Sh}^0(\Gamma) = \Gamma \backslash \mathcal{D}_{p,q}$.

We now recall how these quotients occur as connected components of unitary Shimura varieties. The Hermitian symmetric domain $\mathcal{D}_{p,q}$ corresponds to the $G(\bbR)$-conjugacy class of homomorphisms $\mathrm{Res}_{\bbC/\bbR}\mathbb{G}_{\mathrm{m}}\to G_{\bbR}$. For a compact open subgroup $K\subset G(\mathbb{A}_f)$, the associated \emph{unitary Shimura variety} is defined by the double coset
\[
  \mathrm{Sh}_{K}(G,\mathcal{D}_{p,q})\defeq G(\bbQ)\backslash\bigl(\mathcal{D}_{p,q}\times G(\mathbb{A}_f)\bigr)/K.
\]
It is known that $\mathrm{Sh}_{K}(G,\mathcal{D}_{p,q})$ decomposes into a finite union of unitary modular varieties. More precisely, if $g$ runs over a set of representatives of $G(\bbQ)\backslash G(\mathbb{A}_f)/K$ and  
  $\Gamma_g \defeq G(\bbQ) \cap gKg^{-1}$,
then
\[
  \mathrm{Sh}_{K}(G,\mathcal{D}_{p,q}) = \coprod_{g\in G(\bbQ)\backslash G(\mathbb{A}_f)/K} \Gamma_g\backslash \mathcal{D}_{p,q}.
\]
Thus, when an arithmetic subgroup $\Gamma\subset \U(L)$ arises from a compact open subgroup of $G(\bbA_f)$, the quotient $\mathrm{Sh}^0(\Gamma)$ appears as a connected component of such a Shimura variety; in the rest of the paper we work directly with the arithmetic quotient $\Gamma\backslash\mathcal{D}_{p,q}$.

\begin{rem}
Unitary Shimura varieties form a standard class of PEL-type Shimura varieties modulo central extensions. Their moduli-theoretic and automorphic representations have played an important role in the Langlands program and in the theory of special cycles \cites{Kottwitz1992points,harris2000geometry,shin2011galois,caraiani2017generic,kudla2011special}. In this paper, however, we mainly use their analytic description as arithmetic quotients of type $I_{p,q}$ domains.
\end{rem}

\subsection{Singularities of unitary Shimura varieties}
Our first main result concerns the singularities of a compactification of $\mathrm{Sh}^0(\Gamma)$.
\begin{thm}[{Theorems \ref{thm_classif_sing_disc<-4}, \ref{thm:nonterminal}}]
\label{mainthm:singularity}
  Any unitary modular variety $\mathrm{Sh}^0(\Gamma)$ has a toroidal compactification $\overline{\mathrm{Sh}^0(\Gamma)}^{\Sigma}$ that has at worst canonical singularities for a suitable choice of a regular inversion-free fan $\Sigma$ when $p \geq w_E$ and $(p,q) \neq (2,2),(2,3)$.
  In particular, the same conclusion holds for any unitary Shimura variety $\mathrm{Sh}_{K}(G,\mathcal{D}_{p,q})$.
\end{thm}

Our method uses the Reid--Shepherd-Barron--Tai criterion, following \cites{tai1982kodaira,kondo1993kodaira,Gritsenko2007kodaira} in the Siegel and orthogonal settings, together with refinements specific to the unitary case.
The same analysis also yields the following refinements of the results for $\O^+(2,n)$ \cite{Gritsenko2007kodaira} and $\U(1,n)$ \cite{behrens2012singularities}; see Section \ref{subsec:Application to orthogonal modular varieties}.

Related results in two low-rank unitary settings were obtained
independently by Watson \cite{watson2026singularities}.  He classified the
noncanonical cyclic quotient types for Eisenstein ball
quotients, including their toroidal boundary, and also classified
the possible noncanonical interior types for generalized ball
quotients of type $I_{2,n}$ over
$\mathbb Q(\sqrt{-3})$.
See Section \ref{sec:other_groups} for more precise arguments and discussions.
We will also give a classification of possible non-canonical singularities in Tables \ref{list_canonical_sing} and \ref{list_noncanonical_sing}.

\subsection{Kodaira dimension}
The Kodaira dimension of modular varieties is one of the most important topics in algebraic geometry \cites{tai1982kodaira,freitag1983Siegelsche,mumford1977hirzebruch,kondo1993kodaira,Gritsenko2007kodaira,ma2018kodaira,2025_HMY_Kodaira_dimension_ball_quotients}.
Previous studies, particularly those concerning the general type, have mostly focused on types III and IV, with the exception \cite{2025_HMY_Kodaira_dimension_ball_quotients}.  There has been very little work on type I domains of rank at least 2. 
We claim that almost all such varieties are of general type.
Let     
\begin{equation}\label{def_B(E)}
B(E)\defeq 
\begin{cases}
    5 & \text{if $D \geq 19$ and $D \neq 23$;}\\
    9 & \text{if $D =8, 11, 15, 23$;}\\
    13 & \text{if $D =7$;}\\
    17 & \text{if $D=4$;}\\
    25 & \text{if $D=3$.}
\end{cases}
\end{equation}

\begin{thm}
    \label{mainthm:general type}
    Let $L$ be as above and let $\Gamma\subset \U(L)$ be an arithmetic subgroup.
    Then the unitary modular variety $\mathrm{Sh}^0(\Gamma)$ is of general type if $p \geq w_E$, $(p,q) \neq (2,2),(2,3)$, $p+q\equiv 1 \pmod{w_E}$, and $p+q \geq B(E)$.

\end{thm}

\begin{rem}
The case $p=1$, i.e., ball quotients, was treated in    \cite{2025_HMY_Kodaira_dimension_ball_quotients}.
 Since such unitary groups have rank $1$, their toroidal compactifications are unique and geometrically simpler than those in higher rank; however, reflective divisors must be taken into account \cite{maeda2024reflective}.
\end{rem}

The assumptions in Theorem \ref{mainthm:general type} have different origins. The condition $p\geq w_E$ and the exclusions $(p,q)\neq(2,2),(2,3)$ are used in the singularity analysis, especially in Theorem \ref{mainthm:singularity}. The congruence $p+q\equiv 1 \pmod{w_E}$ and $p+q \geq B(E)$ are used in the construction of low weight cusp forms in Sections \ref{sec:Archimedean packets in the present case} and \ref{sec:The Kodaira dimension}.
\subsection{Orthogonal modular varieties}
\label{subsec:Application to orthogonal modular varieties}

The method developed above also gives an application to orthogonal
modular varieties.  Let $L$ be a lattice of signature $(2,n)$ and let
$\Gamma<\O^+(L)$ be a subgroup of finite index.  We write
$\calF_L(\Gamma)
  \defeq
  \Gamma\backslash\calD_L$
for the associated orthogonal modular variety.

The singularities of toroidal compactifications of orthogonal modular
varieties are known to be canonical if $n\geq 9$ by Gritsenko--Hulek--Sankaran
\cite{Gritsenko2007kodaira}*{Theorem 1.1}.  By refining their estimates for the Reid--Shepherd-Barron--Tai sum,
both in the interior and along the boundary, we obtain the following
improvement of their bound.

\begin{thm}[{Theorem \ref{thm:orthogonal_refinement}}]
\label{mainthm:orthogonal-singularity}
Let $L$ be a lattice of signature $(2,n)$ with $n\geq 6$, and let
$\Gamma<\O^+(L)$ be a subgroup of finite index.  Then there exists a
toroidal compactification $\overline{\calF_L(\Gamma)}$ of $\calF_L(\Gamma)$ having at worst canonical singularities.
\end{thm}

The bound in Theorem \ref{mainthm:orthogonal-singularity} is optimal.
More precisely, in dimension $5$ we construct an orthogonal modular
variety having a noncanonical cyclic quotient singularity of type
\[
  \frac{1}{6}(1,1,1,1,1);
\]
see Section \ref{subsec:An example showing the optimality of the bound}.
Thus the condition $n\geq 6$ cannot in general be improved.
Theorem \ref{mainthm:orthogonal-singularity} and the optimality
statement are proved in Section \ref{sec:other_groups}; see in
particular Theorem \ref{thm:orthogonal_refinement} and the explicit
construction of noncanonical singularities.

We also work on the case of ball quotients; see Theorems \ref{thm_ball_quot_singularity}, \ref{thm:noncanonical-sing-ball-quotient} for precise statements.

\subsection*{Acknowledgment}
 The authors are grateful to Shouhei Ma, Takuya Yamauchi, Klaus Hulek and Matthew P. Watson for their valuable and insightful comments. S.H. is partially supported by JSPS KAKENHI Grant Number 23K12965.
Y.M. is partially supported by the Alexander von Humboldt Foundation through a Humboldt research fellowship, by Deutsche Forschungsgemeinschaft (DFG, German Research
Foundation) through the Collaborative Research Centre TRR 326 \emph{Geometry and Arithmetic
of Uniformized Structures}, project number 444845124, and by Waseda University Grant for Special Research Projects (Project number: 2026R-022).

\section{The Reid--Shepherd-Barron--Tai sum}

\subsection{Reid--Shepherd-Barron--Tai criterion}\label{subsec:RT-criterion}
Given a root of unity $\lambda\in\bbC$, we choose its \emph{exponent}
$\theta\in \bbQ/\bbZ$ such that $\lambda=\exp(2\pi\sqrt{-1}\theta)$.
 Let $\varphi$ denote Euler's totient function, i.e., $\varphi(d) = \#(\bbZ/d\bbZ)^\times$.
Let $V$ be a complex vector space and let $H\subset \GL(V)$ be a finite
subgroup. We recall a numerical criterion for determining when the quotient
$V/H$ has canonical or terminal singularities.

\begin{defn}\label{def:age}
  Let $g\in H$ be of order $m$ and $W\subset V$ be a $g$-stable subspace. Choose a basis of $W$ in which the restriction $g|_W$ is diagonal:
  \[
    g|_W=\mathrm{diag}(\zeta_m^{a_1},\dots,\zeta_m^{a_t}),
    \qquad
    0\le a_j\le m-1
  \]
  where $\zeta_m=\exp(2\pi i/m)$.
  The \emph{age} (or the Reid--Shepherd-Barron--Tai sum) of $g$ on $W$ is
  \[
    \mathrm{RT}_W(g):=\sum_{j=1}^t \Bigl\{\frac{a_j}{m}\Bigr\}\in \bbQ_{\ge 0},
  \]
  where $\{x\}$ denotes the fractional part of a rational number $x$.
  When $W=V$, we simply write $\mathrm{RT}(g)$ instead of $\mathrm{RT}_V(g)$.
\end{defn}

An element $g\in \GL(V)$ is called a \emph{quasi-reflection}
if $g$ has eigenvalues
$(1,\dots,1,\lambda)$ with $\lambda\ne 1$.
We use the following criterion for finite quotient singularities \cites{reid1987young,tai1982kodaira}.

\begin{thm}[Reid--Shepherd-Barron--Tai criterion]\label{thm:Reid--Shepherd-Barron--Tai}
  Assume that $H\subset \GL(V)$ contains no quasi-reflections. Then the quotient
  singularity $(V/H,0)$ is canonical (resp.\ terminal)  if and only if $\mathrm{RT}(g)\ge 1$ (resp.\ $>1$) for every $1\ne g\in H$.
\end{thm}

Let $x$ be a point on a smooth local cover of a toroidal compactification of the unitary modular variety and $T_x$ be the tangent space of $x$.
The stabilizer $\Gamma_x$ of $x$ acts on $T_x$ by a homomorphism $\Gamma_x \rightarrow \GL(T_x)$.
In the applications below, the Reid--Shepherd-Barron--Tai criterion is always applied to the image of the homomorphism $\Gamma_x \rightarrow \GL(T_x)$.
Note that the homomorphism may have a nontrivial kernel, such as scalar matrices.

\subsection{Representations over \texorpdfstring{$E$}{}}\label{subsec:cyclic-reps}

Let $H=\langle g\rangle$ be a cyclic group of order $m$. Over $\bbC$, every irreducible
representation of $H$ is one-dimensional and determined by a character
$g\mapsto \zeta_m^a$. Over $E$, irreducible representations of $H$ are governed by the factorization
of $x^m-1$ in $E[x]$. 
Concretely, every irreducible $E$-representation of $H$ occurs in the $E$-algebra
$E[x]/(x^m-1)$ where $g$ acts by multiplication by $x$, and for each irreducible factor
$f\mid (x^m-1)$ one obtains an irreducible $E$-representation $W_f:=E[x]/(f)$.
We recall the decomposition $x^m-1=\prod_{d\mid m}\Phi_d(x)$ into the cyclotomic polynomials over $\bbQ$.
The behavior of $\Phi_d$ over $E$ is controlled by the inclusion $E\subset \bbQ(\zeta_d)$;
$\Phi_d$ remains irreducible over $E$ unless $E\subset \bbQ(\zeta_d)$, in which case it
splits into two polynomials of degree $\varphi(d)/2$.
More precisely, the following holds.
\begin{lem}[{\cite{behrens2012singularities}*{Subsection 2.1}, \cite{1928_Weisner}*{p.381}}]
  Let $d$ be a positive integer.
  \begin{enumerate}
    \item If $\Phi_d$ is irreducible over $E$, then the set of roots is of the form
      \[
        \{\zeta_d^a \mid a \in (\bbZ/d\bbZ)^\times\}.
      \]
    \item
      The polynomial $\Phi_d$ is reducible over $E$ if and only if $D\mid d$.
  \end{enumerate}
  In the second case, the quadratic character $\omega_E$ associated with $E$ defines a character of $(\bbZ/D\bbZ)^\times$ and the pull back with respect to $(\bbZ/d\bbZ)^\times \rightarrow (\bbZ/D\bbZ)^\times$ defines a character $\omega'_E$ of $(\bbZ/d\bbZ)^\times$.
  Then, we write $\Phi_d^+$ and $\Phi_d^-$ as irreducible polynomials dividing $\Phi_d$, whose roots are given by
    \[
      \{\zeta_d^a  \mid a \in (\bbZ/d\bbZ)^\times,\ \omega_E'(a) = \pm1 \}.
  \]
\end{lem}

\section{Combinatorics and estimates for cyclotomic polynomials}\label{section_auxiliary_lemma}
\subsection{Auxiliary lemmas}
For two finite multisets $A$ and $B$ of $\bbQ/\bbZ$, we consider the following sum
\[
  S(A,B)\defeq \sum_{x \in A,\ y \in B} \left\{ x-y\right\}.
\]
For a polynomial $f$ all of whose roots are roots of unity, let $X_f$ be the multiset in $\bbQ/\bbZ$ of the exponents of the roots of $f$.
By $A_f$ and $B_f$, we mean non-empty sub-multisets of $X_f$ such that $X_f = A_f \sqcup B_f$ as multisets.
By taking the complex conjugate, $S(A_f,B_f) = S(-B_f,-A_f)$.
Let
\[
  c_{\min}(f,E) \defeq \min_{A_f \sqcup B_f = X_f} S(A_f,B_f)
\]
when there exists a nontrivial bipartition.
A key ingredient of the present paper is to evaluate $c_{\min}(f, E)$.
Let $d$ be a positive integer.
Before evaluating $S(A_f, B_f)$, we introduce a numerical lemma.

\begin{lem}
  \label{lem:C(d)}
  Set $C(d) \defeq \varphi(d)/d = \prod_{p \mid d}\left(1-1/p\right)$
  and $F(d) \defeq dC(d)^2-C(d)-2$.
  Then $F(d)>0$ if either of the following conditions holds:
  \begin{itemize}
    \item $d$ is odd and $d\geq5$;
    \item $d$ is even and is divisible by a prime $\ell\geq7$.
  \end{itemize}
\end{lem}
\begin{proof}
  Let $\ell$ be a prime divisor of $d$. Since $C(\ell d)=C(d)$, we have $F(\ell d)-F(d)=(\ell-1)dC(d)^2>0$.
  Let $\ell\nmid d$ be an odd prime. Then $C(\ell d) = (1-1/\ell)C(d)$, and hence
  \[
  F(\ell d)-F(d)
  =d\left(\ell-3+\frac{1}{\ell}\right)C(d)^2+\frac{C(d)}{\ell}>0.
  \]
  Therefore $F(\ell d) > F(d)$.

  Suppose that $d$ is odd and $d\geq5$. If $d$ is divisible by a prime $\ell\geq5$, then
  \[
  F(\ell)=\ell\left(1-\frac{1}{\ell}\right)^2-\left(1-\frac{1}{\ell}\right)-2
  =\ell-5+\frac{2}{\ell}>0
  \]
and for the case $d=3^e$ with $e \geq 2$, we have $F(3^e) = (4 \cdot 3^e - 24)/9 > 0$.
By the above discussions, it concludes that $F(d) > 0$ for any odd $d \geq 5$.

  Suppose that $d$ is even and divisible by a prime $\ell\geq7$. 
  The computation $F(2\ell)=2\ell C(2\ell)^2-C(2\ell)-2
  =(\ell-7)/2+1/\ell>0$ implies
  $F(d)\geq F(2\ell)>0$.
\end{proof}
We first consider the case in which $\Phi_d$ is irreducible.
\begin{lem}\label{lemma_RT_sum_irreducible}
  Assume that $\Phi_d$ is irreducible over $E$.
  If $d \not\in \{1,2,3,4,6\}$,
  we have $c_{\min}(\Phi_d,E)> 1$.
  Moreover, we have
  \[
    c_{\min}(\Phi_d, E) \geq
    \begin{dcases}
      \frac{1}{3} & \text{if $d = 3,6$;}\\
      \frac{1}{2} & \text{if $d=4$.}
    \end{dcases}
  \]
\end{lem}
\begin{proof}
  Set $a = \#A_f, b = \#B_f$ and we may assume $a \leq b$.
  We first evaluate
  \[
    \sum_{y \in B_{\Phi_d}} \left\{ x-y\right\}.
  \]
  Note that $x-y$ are different in $\frac{1}{d}\bbZ/\bbZ$ and $x-y \not \equiv 0 \bmod 1$ when $y$ runs over all $y \in B_{\Phi_d}$.
  Hence, we have
  \[
    \sum_{y \in B_{\Phi_d}} \left\{ x-y\right\} \geq \frac{1}{d} + \frac{2}{d} + \cdots + \frac{b}{d} = \frac{b(b+1)}{2d},
  \]
  and
  \[
    S(A_{\Phi_d},B_{\Phi_d}) = \sum_{x \in A_{\Phi_d}, y \in B_{\Phi_d}} \left\{x-y\right\} \geq \frac{ab(b+1)}{2d}.
  \]

  We now estimate the right-hand side.
  Set $t = \varphi(d)$ and $h(x) = x(t-x)(t+1-x)/2$.
  If $t \leq 3$, then $d = 1,2,3,4,6$.
  In these cases, we have $c_{\min}(\Phi_3,E) = S(\{2/3\}, \{1/3\}) = 1/3,c_{\min}(\Phi_6,E) = S(\{1/6\}, \{5/6\}) = 1/3$, and  $c_{\min}(\Phi_4,E) = S(\{1/4\}, \{3/4\}) = 1/2$.
 We assume $t\geq 4$ in the following.
  By assumption, it suffices to evaluate the minimal values of $h(x)$ with $1 \leq x \leq \lfloor t/2\rfloor$.
  By
    $2h'(x) = 3x^2 -2(2t+1)x + t(t+1) = 3(x-(2t+1)/3)^2-(t^2+t+1)/3$,
  we have $2h'(1) = t^2-3t+1 > 0$ and $2h'(t/2) = -t^2/4<0$.
  It suffices to evaluate $h(1)$ and $h(t/2)$.
  Now, $h(1) = t(t-1)/2, h(t/2) = t^2(t+2)/16$, and
    $h(t/2) - h(1) = t(t-2)(t-4)/16$.
  It follows that $h(1)$ is the minimum for $t\geq 4$.
  We evaluate $(t-1)t/2d$.
  From Lemma \ref{lem:C(d)}, if $d$ satisfies one of the conditions in Lemma \ref{lem:C(d)}, we have
  $t(t-1)/2d = C(d)(d \cdot C(d) -1)/2 = F(d)/2 + 1 > 1$.
  The remaining cases are of the form $d = 2^{e_2} 3^{e_3}5^{e_5}$ for some $e_i$.
  A direct check shows that if
  $F(d) \leq 0$,
  then $d \in \{1,2,3,4,6,8,10,12,18,30\}$.
  The remaining finitely many cases are checked by Algorithm \ref{alg:cmin-exceptional}.
\end{proof}

We next consider the case in which $\Phi_d$ is reducible.
\begin{lem}\label{lemma_RT_sum_reducible}
  Let $d$ be a positive integer such that $\Phi_d$ is reducible over $E$.
  If $d$ does not lie in $\{7,8,12,14,15,20,24,30\}$, we have $c_{\min}(\Phi_d^\pm, E)\geq 1$.
  The equality holds only for $d = 9,18$ with $\#A_{\Phi^\pm_d} = 1$ or $\#B_{\Phi^\pm_d} = 1$.
  Moreover, for $d$ with $\varphi(d) \geq 8$, if $\#A_{\Phi^\pm_d}\geq2$ and $\#B_{\Phi^\pm_d}\geq2$, then
  $S(A_{\Phi^\pm_d},B_{\Phi^\pm_d}) \geq 1$.
  The equality holds only for $d=24$ and $E = \bbQ(\sqrt{-6})$.

\begin{table}[htbp]
  \centering
  \caption{Exceptional values of $c_{\min}(\Phi_d^\pm,E)$ needed below}
  \label{table_list_c<=1}
  \renewcommand{\arraystretch}{1.35}
  \setlength{\tabcolsep}{14pt}
  \begin{tabular}{ccc}
    \toprule
    $d$ & $E$ & $c_{\min}(\Phi_d^\pm,E)$ \\
    \midrule
    $7,14$ & $\bbQ(\sqrt{-7})$ & $4/7$ \\
    \addlinespace
    $8$ & $\bbQ(\sqrt{-2})$ & $1/4$ \\
    $8$ & $\bbQ(\sqrt{-1})$ & $1/2$ \\
    \addlinespace
    $9,18$ & $\bbQ(\sqrt{-3})$ & $1$ \\
    \addlinespace
    $12$ & $\bbQ(\sqrt{-3})$ & $1/2$ \\
    $12$ & $\bbQ(\sqrt{-1})$ & $1/3$ \\
    \addlinespace
    $15,30$ & $\bbQ(\sqrt{-3})$ & $6/5$ \\
    $15,30$ & $\bbQ(\sqrt{-15})$ & $11/15$ \\
    \addlinespace
    $20$ & $\bbQ(\sqrt{-1})$ & $6/5$ \\
    $20$ & $\bbQ(\sqrt{-5})$ & $4/5$ \\
    \addlinespace
    $24$ & $\bbQ(\sqrt{-1})$ & $4/3$ \\
    $24$ & $\bbQ(\sqrt{-2})$ & $7/6$ \\
    $24$ & $\bbQ(\sqrt{-3})$ & $3/2$ \\
    $24$ & $\bbQ(\sqrt{-6})$ & $5/6$ \\
    \bottomrule
  \end{tabular}
\end{table}
\end{lem}
\begin{proof}
  The argument of Lemma \ref{lemma_RT_sum_irreducible} reduces the proof to the following finite set: if $d$ does not lie in
  \[\{7,8,9,11,12,14,15,16,18,20,21,22,24,28,
    30,32,36,38,40,42,48,54,60,66,70,72,78,84,90\}\]
  then $c_{\min}(\Phi_d^\pm, E) > 1$.
  By the finite enumeration described in Algorithm \ref{alg:cmin-exceptional}, we obtain the result.

\end{proof}
  Table \ref{table_list_c<=1} is the complete list of exceptional values of $c_{\min}(\Phi_d^\pm, E)$ needed below.
Let
\[
s_{\min}(f,E) \defeq \min_{\theta}\{S(X_f, \{\theta\}), S(\{\theta\}, X_f)\},
\]
where $\theta$ runs over all exponents of $\OO_E^\times$.

\begin{lem}\label{lem:external-point-bound}
Let $X$ be a finite multiset in $\mathbb Q/\mathbb Z$ of cardinality at least two, and set
\[
c(X)\defeq \min_{X=A\sqcup B}S(A,B),
\]
where $A$ and $B$ run over nonempty submultisets with $X=A\sqcup B$.
Then
\[
\inf_{\theta\in\mathbb Q/\mathbb Z} S(X,\{\theta\})\geq c(X),
\qquad
\inf_{\theta\in\mathbb Q/\mathbb Z} S(\{\theta\},X)\geq c(X).
\]
In particular, $s_{\min}(f,E)\geq c_{\min}(f,E)$ for every $f$.
\end{lem}

\begin{proof}
We prove the first inequality.
The function
\[
F(\theta)\defeq S(X,\{\theta\})=\sum_{\alpha\in X}\{\alpha-\theta\}
\]
is piecewise linear on $\mathbb R/\mathbb Z$, and its only break points occur at points of the support of $X$.
On each connected component of the complement of the support of $X$,
each summand $\{\alpha-\theta\}$ has slope $-1$ as a function of
$\theta$.  Hence $F$ has constant slope $-\#X$. Here $\#X$ denotes the cardinality of the multiset $X$, counted with
multiplicity.
Therefore its infimum is attained at a break point, say $\theta=\alpha_0\in X$.
Removing one copy of $\alpha_0$ from $X$, we obtain $F(\alpha_0)=S(X\setminus\{\alpha_0\},\{\alpha_0\})$,
up to zero contributions from other copies equal to $\alpha_0$.
This is one of the values appearing in the definition of $c(X)$, and hence is at least $c(X)$.

For the second inequality, apply the first inequality to the
multiset $-X$ and $S(\{\theta\}, X) = S(-X, \{-\theta\})$. 
\end{proof}

\begin{lem}\label{lem_RT_sum_A_B_empty}
  Let $f$ be an irreducible factor over $E$ of $\Phi_d$ and assume $\deg(f) \geq 2$.
  If $s_{\min}(f,E) \leq 1$, then $d \in \{3,4,6,7,8,12,14,15,20,24,30\}$.
  Table \ref{table_list_S(X,0)<=1} gives the values of $s_{\min}(f,E)$ for such $d$.
\begin{table}[htbp]
  \centering
  \caption{Small values of $s_{\min}(f,E)$ used below}
  \label{table_list_S(X,0)<=1}
  \renewcommand{\arraystretch}{1.35}
  \setlength{\tabcolsep}{14pt}
  \begin{tabular}{ccc}
    \toprule
    $f$ & $E$ & $s_{\min}(f,E)$ \\
    \midrule
      $\Phi_3$ & $\bbQ(\sqrt{-1})$ & $1/2$ \\
      $\Phi_3$ & $E \neq \bbQ(\sqrt{-1}), \bbQ(\sqrt{-3})$ & $1$ \\
      \hline
      $\Phi_4$ & $\bbQ(\sqrt{-3})$ & $2/3$ \\
      $\Phi_4$ & $E \neq \bbQ(\sqrt{-1}), \bbQ(\sqrt{-3})$ & $1$ \\
      \hline
      $\Phi_6$ & $\bbQ(\sqrt{-1})$ & $1/2$ \\
      $\Phi_6$ & $E \neq \bbQ(\sqrt{-1}), \bbQ(\sqrt{-3})$ & $1$\\
      \hline
      $\Phi_{7}^\pm, \Phi_{14}^\pm$ & $\bbQ(\sqrt{-7})$ & $1$ \\
      \hline
      $\Phi_8^\pm$ & $\bbQ(\sqrt{-1})$ & $3/4$ \\
      $\Phi_8^\pm$ & $\bbQ(\sqrt{-2})$ & $1/2$ \\
      \hline
      $\Phi_{12}^\pm$ & $\bbQ(\sqrt{-1})$ & $1/2$ \\
      $\Phi_{12}^\pm$ & $\bbQ(\sqrt{-3})$ & $2/3$ \\
      \hline
      $\Phi_{15}^\pm, \Phi_{30}^\pm$ & $\bbQ(\sqrt{-15})$ & $1$ \\
      \hline
      $\Phi_{20}^\pm$ & $\bbQ(\sqrt{-1})$ & $2$ \\
      $\Phi_{20}^\pm$ & $\bbQ(\sqrt{-5})$ & $1$ \\
      \hline
      $\Phi_{24}^\pm$ & $\bbQ(\sqrt{-1})$ & $3/2$ \\
      $\Phi_{24}^\pm$ & $\bbQ(\sqrt{-2})$ & $2$ \\
      $\Phi_{24}^\pm$ & $\bbQ(\sqrt{-3})$ & $5/3$ \\
      $\Phi_{24}^\pm$ & $\bbQ(\sqrt{-6})$ & $1$\\\hline

  \end{tabular}
\end{table}
\end{lem}
\begin{proof}
    By Lemma \ref{lem:external-point-bound}, we have
  $s_{\min}(f,E) \geq c_{\min}(f,E)$.
    Hence, by Lemmas \ref{lemma_RT_sum_irreducible} and \ref{lemma_RT_sum_reducible}, it suffices to consider the polynomials $\Phi_d$ or $\Phi_d^\pm$ for finitely many $d \in \{3,4,6,7,8,9,12,14,15,18,20,24,30\}$.
    The values in Table \ref{table_list_S(X,0)<=1} then follow from the finite enumeration described in Algorithm \ref{alg:smin-exceptional}.
\end{proof}

\section{Singularities at interior points}

\subsection{Unitary group action on the tangent space}\label{subsec:unitary-action}

We now apply the criterion (Theorem \ref{thm:Reid--Shepherd-Barron--Tai}) to unitary modular varieties.
Since the stabilizer $\Gamma_x$ of a point
$x\in \mathcal{D}_{p,q}$ is finite, $\mathrm{Sh}^0(\Gamma)$ has only finite quotient singularities.

From the realization \eqref{def_realization_D_p,q}, each point $x \in \mathcal{D}_{p,q}$ can be regarded as a negative definite subspace of $L \otimes \bbC$.
We denote by $V_{x}^-$ the negative definite subspace and by $V_{x}^+$ the orthogonal complement.
When $x$ is clear from the context, we abbreviate the subscript $x$.
The unitary groups $\U(V_{x}^+)$ and $\U(V_{x}^-)$ associated with $V_{x}^+$ and $V_{x}^-$ are isomorphic to $\U(p)$ and $\U(q)$, respectively.
After
conjugation we may assume that $\Gamma_x\subset \U(p)\times \U(q)$.
The holomorphic
tangent space at $x$ is canonically identified with
\[
  T_x\mathcal{D}_{p,q}\;\cong\; \Hom_\bbC(V_{x}^-,V_{x}^+)\;\cong\; \bbC^p\otimes (\bbC^q)^\vee.
\]
If $g\in \Gamma_x$ maps to $(a,d)\in \U(p)\times \U(q)$, then the action is given by
  $g\cdot X = aXd^{-1}$ for $X\in \Hom(\bbC^q,\bbC^p)$.
Consequently, if the eigenvalues of $a$ (resp.\ $d$) are $\alpha_1,\dots,\alpha_p$
(resp.\ $\delta_1,\dots,\delta_q$), then the eigenvalues of $g$ on $T_x\mathcal{D}_{p,q}$ are $\alpha_i\delta_j^{-1}$ for $1\le i\le p,\; 1\le j\le q$.

Following \cite{reid1987young}, it is convenient to encode
the Reid--Shepherd-Barron--Tai sum for $T_x\mathcal{D}_{p,q}$ by the difference of exponents. Namely, write
$\alpha_i=\exp(2\pi i x_i)$ and $\delta_j=\exp(2\pi i y_j)$ with $x_i,y_j\in \bbQ/\bbZ$.
Then the age of $g$ on $T_x\mathcal{D}_{p,q}$ equals
\begin{equation}\label{eq:RT-unitary}
  \mathrm{RT}(g)=\sum_{i=1}^p\sum_{j=1}^q \{x_i-y_j\}.
\end{equation}
Let $X_g^+$ and $X_g^-$ denote the multisets of exponents of eigenvalues of $g$ on $V_x^+$ and $V_x^-$, respectively. Then \eqref{eq:RT-unitary} can be written as
$\mathrm{RT}(g)=S(X_g^+,X_g^-)$.

Since $g\in \U(L)$, its action on $L\otimes_{\OO_E}E$ is $E$-linear. In particular, its characteristic polynomial $p_g(T) = \det(T \cdot 1_{p+q} - g)$ has coefficients in $E$.
Since $g$ is of finite order, $p_g(T)$ is a product of irreducible factors over $E$.
By this observation, the computation of singularities is reduced to the auxiliary lemmas of Section \ref{section_auxiliary_lemma}.
Combining Theorem \ref{thm:Reid--Shepherd-Barron--Tai} with \eqref{eq:RT-unitary}, the canonicity of
singularities of $\mathrm{Sh}^0(\Gamma)$ at an interior point reduces to estimating $S(X_g^+,X_g^-)$.

Below, we record that the stabilizers at
interior points act on the tangent space without quasi-reflections.
This follows from the assumption $p,q\geq2$.
It is known that if an arithmetic subgroup contains a quasi-reflection, then it is of type $\O^+(2,n)$ or $\U(1,n)$ \cite{meschiari1972reflections}.
This applies to our case, but to include the statement for $(p,q) = (2,2)$, we give a proof of the absence of quasi-reflections in $\U(L)$.

\begin{prop}\label{prop:no-qref}
  No nontrivial element of a finite stabilizer $\Gamma_x\subset \U(V_{x}^+)\times \U(V_{x}^-)$
  acts as a quasi-reflection on $T_x\mathcal{D}_{p,q}$.
\end{prop}

\begin{proof}
  Let $g=(a,d)\in \U(V_{x}^+)\times \U(V_{x}^-)$ be of finite order and suppose that $g^k$ acts as a quasi-reflection
  on $T_x\mathcal{D}_{p,q}$ for some $k\ge 1$.  Write the eigenvalues of $a$ and $d$ as above.
  Then the eigenvalues of $g^k$ on $T_x\mathcal{D}_{p,q}$ are $(\alpha_i\delta_j^{-1})^k=\alpha_i^{k}(\delta_j^{k})^{-1}$.
  The assumption on the quasi-reflection means that all but one of these $pq$ numbers are equal to $1$.
  Thus there exist indices $(i_0,j_0)$ such that $\alpha_i^{k}=\delta_j^{k}$ for all $(i,j)\neq(i_0,j_0)$, and
    $\alpha_{i_0}^{k}\neq \delta_{j_0}^{k}$.
  Choose $i_1\neq i_0$ and $j_1\neq j_0$, which are possible since $p,q\ge 2$.
  Applying the equalities above to the three pairs $(i_0,j_1)$, $(i_1,j_0)$, and $(i_1,j_1)$ gives
    $\alpha_{i_0}^{k}=\delta_{j_1}^{k}$,
    $\alpha_{i_1}^{k}=\delta_{j_0}^{k}$, and 
    $\alpha_{i_1}^{k}=\delta_{j_1}^{k}$.
  Hence $\delta_{j_0}^{k}=\alpha_{i_1}^{k}=\delta_{j_1}^{k}=\alpha_{i_0}^{k}$, contradicting
  $\alpha_{i_0}^{k}\ne\delta_{j_0}^{k}$.
  Therefore no such $k$ exists, and in particular no power of $g$ acts as a quasi-reflection.
\end{proof}

\begin{rem}
  When $(p,q)=(2,2)$ one has the well-known exceptional isomorphism of real Lie groups
    $\SU(2,2)/\{\pm 1\}  \cong \SO^+(2,4)$,
  and consequently the symmetric domain of type $I_{2,2}$ for $\SU(2,2)$ can also be viewed as the
  type IV domain attached to a quadratic space of signature $(2,4)$.
  Although $\O^+(2,4)$ may contain reflections fixing a hyperplane in the orthogonal model, such elements either lie outside the identity component or do not
  act as quasi-reflections on the holomorphic tangent space of $\mathcal{D}_{2,2}$ by Proposition \ref{prop:no-qref}.
\end{rem}

\begin{cor}\label{cor:branch-codimge2}
  The uniformization map $\mathcal{D}_{p,q}\to \mathrm{Sh}^0(\Gamma)$ is unramified in codimension $1$.
\end{cor}

\begin{proof}
  A finite quotient map is branched along a divisor precisely when some stabilizer contains a quasi-reflection.
  This cannot happen by Proposition \ref{prop:no-qref}.
\end{proof}
For ball quotients, where $\min\{p,q\}=1$, quasi-reflections do occur and correspond
to complex reflections fixing a divisor \cites{behrens2012singularities,maeda2024reflective}.

\subsection{Ages of unitary group actions}
\label{sec:ages_unitary_actions}
In this subsection, we prove that, outside finitely many low-rank cases, unitary modular varieties have at worst canonical singularities.

\begin{lem}\label{lemma_S_AA_BB}
  Let $f$ and $f'$ be irreducible polynomials dividing $\Phi_d$ and $\Phi_{d'}$ with $\deg(f) \geq 2$ and $\deg(f') \geq 2$, respectively.
  Decompose $X_f=A_f\sqcup B_f$ and $X_{f'}=A_{f'}\sqcup B_{f'}$
  with all four sets $A_f,B_f,A_{f'},B_{f'}$ nonempty.  Then $S(A_f \sqcup A_{f'}, B_{f} \sqcup B_{f'}) \geq 1$.
  The equality holds only if $f=f'$ with $\deg(f) =2$.
\end{lem}
\begin{proof}
  Note first that
  \[
    S(A_f \sqcup A_{f'}, B_f \sqcup B_{f'}) \geq S(A_f, B_f) + S(A_{f'},B_{f'}).
  \]
  If either $f$ or $f'$ is outside the exceptional set appearing in Lemmas \ref{lemma_RT_sum_irreducible} and \ref{lemma_RT_sum_reducible}, the assertion follows immediately from those lemmas. It remains to check finitely many exceptional pairs. For these pairs, we enumerate all decompositions $X_f=A_f\sqcup B_f$ and $X_{f'}=A_{f'}\sqcup B_{f'}$ and compute 
$S(A_f\sqcup A_{f'},B_f\sqcup B_{f'})$.
  This finite enumeration, described in Algorithm \ref{alg:exceptional-pair}, gives the claimed lower bound and shows that equality occurs only when $f=f'$ and $\deg f = 2$.
\end{proof}

\begin{lem}\label{lemma_S_XA_B}
  Let $f$ and $f'$ be irreducible polynomials dividing $\Phi_d$ and $\Phi_{d'}$ with $\deg(f) \geq 2$ and $\deg(f') \geq 2$, respectively.
  If $S(A_f \sqcup X_{f'}, B_f)$
    or $S(A_f,B_f \sqcup X_{f'})$ is at most $1$, then $\deg(f) = 2$.
  In particular, $\#A_f = \#B_f = 1$.
\end{lem}
\begin{proof}
  We consider the sum $S(A_f \sqcup X_{f'}, B_f)$; the other case is similar.
  We have
  \[
    S(X_{f'},B_f) \geq \#B_f \cdot c_{\min}(f',E) \geq c_{\min}(f',E) \geq
    \begin{cases}
      1/4 & \text{if $E = \bbQ(\sqrt{-2})$;}\\
      1/3 & \text{if $E \neq \bbQ(\sqrt{-2})$.}
    \end{cases}
  \]
  In particular, if $S(A_f, B_f) >3/4$, we have
  $S(A_f \sqcup X_{f'}, B_f)> 1$.
  By Lemmas \ref{lemma_RT_sum_irreducible} and \ref{lemma_RT_sum_reducible}, a polynomial $f$ with $S(A_f, B_f) \leq 3/4$ and $\deg(f) > 2$ is either $\Phi_7^{\pm}$, $\Phi_{14}^\pm, \Phi_{15}^\pm$ or $\Phi_{30}^\pm$ and then $E = \bbQ(\sqrt{-7})$ or $\bbQ(\sqrt{-15})$ and $E \neq \bbQ(\sqrt{-2})$.
  Since $c_{\min}(X_{\Phi_{15}^\pm}, \bbQ(\sqrt{-15})) = c_{\min}(X_{\Phi_{30}^\pm}, \bbQ(\sqrt{-15})) = 11/15 > 2/3$, the sum is greater than $1$ and hence it suffices to consider the case $f = \Phi_7^\pm, \Phi_{14}^\pm$.
  By $c_{\min}(X_{\Phi_{7}^\pm}, \bbQ(\sqrt{-7})) = c_{\min}(X_{\Phi_{14}^\pm}, \bbQ(\sqrt{-7})) = 4/7$, as $c_{\min}(f',\bbQ(\sqrt{-7})) \leq 3/7$, we may assume $f' = \Phi_3$ or $\Phi_6$.
 For each of the remaining pairs $f\in\{\Phi_7^\pm,\Phi_{14}^\pm\}$ and $f'\in\{\Phi_3,\Phi_6\}$,
over $E=\Q(\sqrt{-7})$, direct evaluation gives
\[
  \min_{\substack{X_f=A_f\sqcup B_f\\
                  A_f\neq\emptyset,\ B_f\neq\emptyset}}
  \min\bigl\{
    S(A_f\sqcup X_{f'},B_f),
    S(A_f,B_f\sqcup X_{f'})
  \bigr\}
  =\frac97>1.
\]
This excludes $\deg(f)>2$ and proves the lemma.
\end{proof}

\begin{lem}\label{lemma_S_XX}
  Let $f$ and $f'$ be irreducible polynomials dividing $\Phi_d$ and $\Phi_{d'}$ with $\deg(f) \geq 2$ and $\deg(f') \geq 2$, respectively.
  Then,
  $S(X_f, X_{f'}) \geq 1$.
  The equality holds only if $f=f'$ with $\deg(f)=2$.
\end{lem}
\begin{proof}
  Since $\deg(f), \deg(f') \geq 2$, if $c_{\min}(f,E)$ or $c_{\min}(f',E)$ is greater than $1/2$, the claim holds.
  Hence, we may assume $f,f' = \Phi_3,\Phi_4,\Phi_6,\Phi_8^\pm,\Phi_{12}^\pm$ by Lemma \ref{lemma_RT_sum_reducible}.
  Then the statement follows from the finite enumeration described in Algorithm \ref{alg:exceptional-pair}.
\end{proof}

With the above preliminaries, we compute the singularities.
Let $\gamma$ be an element of the stabilizer $\Gamma_x$ and $p_\gamma(T) = \det(T \cdot 1_{p+q} - \gamma)$ be the characteristic polynomial.
For each occurrence of an irreducible factor $f$ of
$p_\gamma(T)$, let $W_f$ be an abstract copy of the
corresponding irreducible $E[\gamma]$-module.
We choose $\gamma$-equivariant complex embeddings of
$W_{f,\bbC}:=W_f\otimes_E\bbC$ into $L\otimes_E\bbC$
such that their images form a direct-sum decomposition
compatible with $V_x^+\oplus V_x^-$.
Such embeddings can be chosen by taking bases of the
positive and negative parts of every complex eigenspace
and distributing the eigenvalue occurrences among the
copies of each irreducible factor.
Set
$A_f:=\{\text{exponents on }W_{f,\bbC}\cap V_x^+\}$
and
$B_f:=\{\text{exponents on }W_{f,\bbC}\cap V_x^-\}$.
Then $X_f=A_f\sqcup B_f$; either part may be empty.
Repeated factors are counted separately.
The chosen complex embeddings need not be defined over $E$.
\begin{prop}\label{prop_two_nonlinear_factor}
  If $p_\gamma(T) = \det(T \cdot 1_{p+q} - \gamma)$ has at least two irreducible nonlinear 
  factors $f_1$ and $f_2$ over $E$, then
  $\mathrm{RT}(\gamma) \geq 1$.
  The equality holds only if $(p,q) = (2,2)$ and $f_1=f_2$.
\end{prop}
\begin{proof}
  Let $f_1$ and $f_2$ be such nonlinear factors.
  Consider the representation on $L \otimes E$ of the cyclic group generated by $\gamma$.
  The polynomials $f_i$ define subrepresentations $W_i$ of $L \otimes E$.
  In the following, containment statements such as $W_i\subset V^\pm$ mean the corresponding statement for the complexification $W_{i,\bbC}$ inside $L\otimes_{E}\bbC$.
  Recall that the point $x$ defines subspaces $V^+$ and $V^-$ of $L \otimes \bbC$.
  By relabeling $f_1$ and $f_2$ if necessary, the problem is reduced to the following cases:
  \begin{enumerate}
    \item Both $W_i$ are not contained in $V^{\pm}$.
    \item $W_1$ is contained in $V^+$ or $V^-$, but the other $W_2$ is not.
    \item $W_1$ is contained in $V^+$ and $W_2$ is contained in $V^-$.
    \item Both $W_i$ are contained in the same $V^+$ or $V^-$.
  \end{enumerate}
  For the first and the third cases, the statements follow from Lemma \ref{lemma_S_AA_BB} and Lemma \ref{lemma_S_XX}.

  We consider the second case.
  Suppose that $W_1 \subset V^+$.
  Let $A$ (resp.\ $B$) be the set of exponents of the eigenvalues in $W_2 \cap V^+$ (resp.\ $W_2 \cap V^-$).
  The proof for the case $W_1 \subset V^-$ is similar.
  By definition, $\mathrm{RT}(\gamma) \geq S(A \sqcup X_{f_1}, B)$.
  By Lemma \ref{lemma_S_XA_B} and $q \geq 2$, we may assume $\#
  B = 1$ and $W_2 \cap V^- \neq V^-$.
  If there exists an eigenvalue on $V^- / (W_2 \cap V^-)$ that is not a unit in $\OO_E$, there exists a polynomial $f_3$ dividing $p_\gamma(T)$ such that the eigenvalue is a root of $f_3$.
  Define $W_3$ similarly.
  If $W_3$ is not contained in $V^-$, then the pair $(W_2, W_3)$ is reduced to the first case by $W_3 \cap V^- \neq 0$. 
  If $W_3$ is contained in $V^-$, then the pair $(W_1, W_3)$ is reduced to the third case.
  Hence, we may assume that all the eigenvalues on $V^- / (W_2 \cap V^-)$ are units in $\OO_E$.
  For the exponent $a$ of such a unit, we have
  \begin{align*}
    \mathrm{RT}(\gamma) &\geq S(A \sqcup X_{f_1}, B \sqcup \{a\}) > S(A,B) + S(X_{f_1}, B \sqcup \{a\})\\
    &\geq c_{\min}(f_2, E) + c_{\min}(f_1,E) + s_{\min}(f_1,E).
  \end{align*}
  Here, we use $S(A,\{a\}) > 0$ by definition of $A$ and $a$.
  By $c_{\min}(f,E) \geq 1/4$ and $s_{\min}(f,E) \geq 1/2$, we obtain $\mathrm{RT}(\gamma) > 1$.

  It remains to consider the last case.
  It suffices to treat the case $W_1,W_2\subset V^+$.
  By the same argument as in the second case, we may assume that all the eigenvalues on $V^-$ are units in $\OO_E$.
  Then a direct calculation gives
  \[
    \mathrm{RT}(\gamma) \geq S(X_{f_1} \sqcup X_{f_2}, \Theta^-) \geq q \cdot (s_{\min}(f_1,E) + s_{\min}(f_2,E)) \geq q \cdot 2 \cdot 1/2 \geq 2 > 1,
  \]
  where $\Theta^-$ denotes the multiset of exponents of eigenvalues on $V^-$.
  This completes the proof.
\end{proof}

\begin{prop}\label{prop_one_nonlinear_factor}
  Suppose the irreducible decomposition $p_\gamma(T) = \prod_i f_i$ with $\deg(f_1)\geq 2$, $\deg(f_i)=1$ for $i\neq 1$ and $(p,q) \neq (2,2), (2,3)$.
  Then $\mathrm{RT}(\gamma) \geq 1$.
\end{prop}

\begin{proof}
Let $f$ be the unique nonlinear irreducible factor.
If $A_f=\emptyset$ or $B_f=\emptyset$, the opposite sign
contains at least two unit eigenvalues. Hence
\[
  \mathrm{RT}(\gamma)\geq 2 \cdot s_{\min}(f,E)\geq1
\]
by Lemma \ref{lem_RT_sum_A_B_empty}.
Otherwise, $\mathrm{RT}(\gamma)\geq c_{\min}(f,E)$.
It remains to consider the independently specified finite
input $\mathcal F_{\mathrm{one}}$ in Appendix
\ref{app:finite-computations}.

The finite enumeration in Algorithm \ref{alg:low-age-types} shows that the age strictly less than one
can occur only for $(p,q)=(2,2)$ or $(2,3)$.
\end{proof}

\begin{lem}\label{lem:unit-eigenvalues-interior}
Let $A$ and $B$ be multisets in $\frac{1}{w_E}\bbZ/\bbZ$ with $\#A=p$ and $\#B=q$. 
Suppose that the multiset $A \sqcup B$ contains at least two different rational numbers.
Then $S(A,B) \geq p/w_E$.
\end{lem}
\begin{proof}
For $\theta\in B$, set $F_A(\theta)\defeq S(A,\{\theta\})=\sum_{a\in A}\{a-\theta\}$.
If $F_A(\theta)=0$ for some $\theta \in B$, then every element of $A$ is equal to $\theta$.
By the assumption, there exists $\theta'\in B$ such that $\theta \neq \theta'$.
Hence, we have $S(A,B) \geq p\{\theta-\theta'\} \geq p/w_E$.
If $F_A(\theta)$ is nonzero for any $\theta \in B$, then $S(A,B) \geq \sum_\theta F_A(\theta)\geq q/w_E \geq p/w_E$, which implies the statement.
\end{proof}

\subsection{Canonical singularities}
Now, we study when $\mathrm{Sh}^0(\Gamma)$ has canonical or, moreover, terminal singularities.

\begin{thm}\label{thm_classif_sing_disc<-4}
Any unitary modular variety $\mathrm{Sh}^0(\Gamma)$ has at worst canonical singularities if $p\geq w_E$ and $(p,q) \neq (2,2), (2,3)$.
\end{thm}
\begin{proof}
Take $\gamma\in \Gamma_x$ whose action on the tangent space is nontrivial. 
By Proposition \ref{prop:no-qref}, it suffices to show $\mathrm{RT}(\gamma) \geq 1$.
When the characteristic polynomial $p_\gamma$ has a root that does not lie in $\OO_E^\times$, the statement follows from Propositions \ref{prop_two_nonlinear_factor} and \ref{prop_one_nonlinear_factor}. 
When all eigenvalues are units of $\OO_E$, the desired inequality follows from Lemma \ref{lem:unit-eigenvalues-interior}.
\end{proof}

\begin{rem}
For $p=2$ and $E=\mathbb Q(\sqrt{-3})$, the possible
noncanonical cyclic quotient types in the interior were
independently classified by Watson
\cite{watson2026singularities}*{Theorem 7.2.5}.
For $q\geq 6$, his list consists of
\[
  \frac{1}{3}(1,1),\qquad
  \frac{1}{6}(1,1),\qquad
  \frac{1}{6}(1,1,1,1),
\]
up to adjoining zero weights.  This agrees with the
corresponding $p=2$, $E=\mathbb Q(\sqrt{-3})$ entries in
Tables \ref{list_canonical_sing}
and \ref{list_noncanonical_sing} below.
Watson's result concerns the interior of the type
$I_{2,q}$ case,
whereas our classification treats arbitrary
imaginary quadratic fields and is combined below with the analysis of singularities on the boundaries.
\end{rem}

\begin{ex}\label{ex:exceptional-interior}
We give explicit realizations of the noncanonical interior
singularities occurring for $(p,q)=(2,2)$ and $(2,3)$.
Let $\zeta_3\defeq\exp(2\pi\sqrt{-1}/3)$
and put
\[
H\defeq
\begin{pmatrix}
0&-\sqrt{-D}\\
\sqrt{-D}&0
\end{pmatrix},
\qquad
A\defeq
\begin{pmatrix}
-1&-1\\
1&0
\end{pmatrix}.
\]
Then $A^*HA=H$ and $A^3=I_2$.

For $2\leq p\leq q$, let $L_{p,q}\defeq\OO_E^{p+q}$
and endow it with the Hermitian form $h_{p,q}
\defeq
H
\oplus
\langle 1\rangle^{p-1}
\oplus
\langle -1\rangle^{q-1}$.
Thus $h_{p,q}$ has signature $(p,q)$. Define $\gamma_{p,q}
\defeq
\diag\bigl(A,-1_{p+q-2}\bigr)
\in \U(L_{p,q})$.
The element $\gamma_{p,q}$ has order six.
Let $v_-\defeq\zeta_3e_0+e_1$ and $v_+\defeq\zeta_3^2e_0+e_1$.
A direct calculation gives $Av_-=\zeta_3v_-$, $
Av_+=\zeta_3^2v_+$,
and $h_{p,q}(v_-,v_-)=-\sqrt{3D}<0$, $h_{p,q}(v_+,v_+)=\sqrt{3D}>0$.
Hence
\[
V^-:=
\bbC v_-
\oplus
\bigoplus_{j=p+1}^{p+q-1}\bbC e_j
\]
is a negative-definite $q$-dimensional subspace, whereas
\[
V^+:=
\bbC v_+
\oplus
\bigoplus_{j=2}^{p}\bbC e_j
\]
is its 
orthogonal complement.
Therefore $x_{p,q}:=V^-\in\mathcal D_{p,q}$
is an interior point fixed by $\gamma_{p,q}$.
The multisets of exponents of the eigenvalues of $\gamma_{p,q}$
on $V^+$ and $V^-$ are
\[
X^+_{\gamma_{p,q}}
=
\left\{
\frac23,
\smash{\underbrace{\frac12,\ldots,\frac12}_{p-1}}
\right\},
\qquad
X^-_{\gamma_{p,q}}
=
\left\{
\frac13,
\smash{\underbrace{\frac12,\ldots,\frac12}_{q-1}}
\right\}.
\]
Consequently,
\[
\mathrm{RT}(\gamma_{p,q})
=
S\bigl(
X^+_{\gamma_{p,q}},
X^-_{\gamma_{p,q}}
\bigr)=
\frac13
+\frac{q-1}{6}
+\frac{p-1}{6}=
\frac{p+q}{6}.
\]
For $(p,q)=(2,2), (2,3)$, the singularities are of type
\[
\frac{1}{6}(2,1,1,0),\qquad \frac{1}{6}(2,1,1,1,0,0),
\]
respectively.
Hence these are noncanonical singularities.
\end{ex}

The preceding results determine the possible nonterminal singularities at interior points.
The corresponding interior quotient types are recorded below.
The following tables list the types of 
a
generator $\gamma$ with $0<\mathrm{RT}(\gamma)\leq1$.

\renewcommand{\arraystretch}{1.2}
\begin{center}
\refstepcounter{table}\label{list_canonical_sing}
\textbf{Table \thetable.} 
List of possible nonterminal but canonical cyclic quotient singularities at interior points.
The entry ``General'' means that the type is not restricted to a special imaginary quadratic field.
\par\smallskip
\begin{tabular}{c|c|c}
    $(p,q)$ & Singularities & $E$ \\
    \hline
    $(2,2)$ & $\frac{1}{3}(2,1,0,0)$ & General \\
    $(2,2)$ & $\frac{1}{4}(2,1,1,0)$ & General \\
    $(2,2)$ & $\frac{1}{4}(3,1,0,0)$ & $\bbQ(\sqrt{-1}), \bbQ(\sqrt{-2})$ \\
    $(2,2)$ & $\frac{1}{4}(1,1,1,1)$ & $\bbQ(\sqrt{-1}), \bbQ(\sqrt{-2})$ \\
    $(2,2)$ & $\frac{1}{6}(3,2,1,0)$ & $\bbQ(\sqrt{-3})$ \\
    $(2,2)$ & $\frac{1}{6}(5,1,0,0)$ & $\bbQ(\sqrt{-3})$ \\
    $(2,2)$ & $\frac{1}{8}(4,3,1,0)$ & $\bbQ(\sqrt{-1})$ \\
    $(2,2)$ & $\frac{1}{12}(5,3,3,1)$ & $\bbQ(\sqrt{-6})$ \\
    $(2,2)$ & $\frac{1}{12}(6,5,1,0)$ & $\bbQ(\sqrt{-3})$ \\
    $(2,4)$ & $\frac{1}{6}(2,1,1,1,1,0,0,0)$ & General \\
    $p=2$ & $\frac{1}{2}(1,1,\underbrace{0,\ldots,0}_{2q-2})$ & General \\
    $p=2$ & $\frac{1}{6}(2,2,1,1,\underbrace{0,\ldots,0}_{2q-4})$ & $\bbQ(\sqrt{-3})$ \\
    $p=2, q\geq2$ & $\frac{1}{8}(3,3,1,1,\underbrace{0,\ldots,0}_{2q-4})$ & $\bbQ(\sqrt{-2})$ \\
    $p=2, q\geq2$ & $\frac{1}{12}(5,5,1,1,\underbrace{0,\ldots,0}_{2q-4})$ & $\bbQ(\sqrt{-1})$ \\
    $p=2,4$ or $q=4$ & $\frac{1}{4}(1,1,1,1,\underbrace{0,\ldots,0}_{pq-4})$ & $\bbQ(\sqrt{-1})$ \\
    $p=2, q\geq3$ or $p=3,6$ or $q=6$ & $\frac{1}{6}(1,1,1,1,1,1,\underbrace{0,\ldots,0}_{pq-6})$ & $\bbQ(\sqrt{-3})$ \\
    $(3,3)$ & $\frac{1}{6}(2,1,1,1,1,0,0,0,0)$ & General \\
    $p=3$ or $q=3$ & $\frac{1}{3}(1,1,1,\underbrace{0,\ldots,0}_{pq-3})$ & $\bbQ(\sqrt{-3})$
\end{tabular}
\end{center}

\renewcommand{\arraystretch}{1.2}
\begin{center}
\refstepcounter{table}\label{list_noncanonical_sing}
\textbf{Table \thetable.} List of possible noncanonical cyclic quotient singularities at interior points.
\par\smallskip
\begin{tabular}{c|c|c}
    $(p,q)$ & Singularities & $E$ \\
    \hline
    $(2,2)$ & $\frac{1}{6}(2,1,1,0)$ & General \\
    $(2,3)$ & $\frac{1}{6}(2,1,1,1,0,0)$ & General \\
    $p=m$ or $q=m$ with $m\leq3$ & $\frac{1}{4}(\underbrace{1,\ldots,1}_m,\underbrace{0,\ldots,0}_{pq-m})$ & $\bbQ(\sqrt{-1})$ \\
    $p=m$ or $q=m$ with $m\leq3$ & $\frac{1}{6}(\underbrace{1,\ldots,1}_m,\underbrace{0,\ldots,0}_{pq-m})$ & $\bbQ(\sqrt{-3})$\\
    $p=2$
& $\frac{1}{3}(1,1,\underbrace{0,\ldots,0}_{2q-2})$
& $\bbQ(\sqrt{-3})$\\
$p\mid4$ or $q\mid4$
& $\frac{1}{6}(1,1,1,1,\underbrace{0,\ldots,0}_{pq-4})$
& $\bbQ(\sqrt{-3})$\\
$p=5$ or $q=5$
& $\frac{1}{6}(1,1,1,1,1,\underbrace{0,\ldots,0}_{pq-5})$
& $\bbQ(\sqrt{-3})$
\end{tabular}
\end{center}

\section{Singularities at boundary points}\label{sec:boundary-singularities}
\subsection{Toroidal compactifications}
Here, we recall the basic facts of toroidal compactifications for $\mathrm{Sh}^0(\Gamma)$ \cites{ash2010smooth,kasparian2021toroidal}.
The Baily--Borel compactification $\mathrm{Sh}^0(\Gamma)^*$ is obtained by adding the
$\Gamma$-orbits of rational boundary components of $\mathcal{D}_{p,q}$; it is a normal projective
variety and is minimal among algebraic compactifications of $\mathrm{Sh}^0(\Gamma)$.
Toroidal compactifications are refinements of $\mathrm{Sh}^0(\Gamma)^*$ obtained by replacing neighborhoods of cusps by torus embeddings.

In the present type $I_{p,q}$ case, rational boundary components are indexed by $\Gamma$-orbits
of totally isotropic $E$-subspaces $I\subset L\otimes_{\OO_E}E$.
Equivalently, they correspond to $\bbQ$-parabolic subgroups of the $\bbQ$-group
$G(\Q) = \U(L)(\Q)$; 
for fixed $s=\dim_E I$ (with $1\le s\le \min\{p,q\}$),
the associated maximal parabolic is the stabilizer $P_I\defeq\Stab(I)$.
For the explicit description of $P_I$, see
\cite{kasparian2021toroidal}*{Section 4.1}.

Let us fix one totally isotropic subspace $I$.
Let $P=P_I$ be such a $\bbQ$-parabolic subgroup and set $\Gamma_P\defeq\Gamma\cap P$.
Write $N_P$ for the unipotent radical of $P$ and $U_P\defeq \mathrm{Z}(N_P)$ for its center.
Choose a Levi subgroup $L_P$ of $P$. As in the refined Langlands decomposition,
\[
 P=N_P\rtimes L_P,
 \qquad L_P\cong\GL_E(I)\times G_h(P),
\]
and there is a central exact sequence
\[
 1\longrightarrow U_P\longrightarrow N_P
 \longrightarrow V_P\longrightarrow1,
\]
where $V_P:=N_P/U_P$ is abelian. 
The group $G_h(P)$ acts on the smaller Hermitian symmetric
domain $F_P\cong\mathcal D_{p-s,q-s}$; see  \cite{kasparian2021toroidal}*{Section 2.2}.
This induces a (real-analytic) horospherical decomposition of the symmetric domain
\begin{equation}\label{eq:horospherical}
  \mathcal{D}_{p,q}\ \cong\ \Lie(U_P)\times \Lie(V_P)\times C_P\times F_P,
\end{equation}
where $C_P\subset \sqrt{-1}\Lie(U_P)$ is an open self-adjoint (strongly convex) cone.
Set $\Upsilon_P:=\Gamma\cap U_P$ and
$T_P(\bbC):=(\Upsilon_P\otimes_{\bbZ}\bbC)/\Upsilon_P$.
Let $\mathcal B_P$ denote the base of the holomorphic
Siegel-domain fibration associated with $P$.
Its underlying real-analytic manifold is identified with
$\Lie(V_P)\times F_P$ via the horospherical decomposition.
We equip $\mathcal B_P$ with the complex structure induced
by the holomorphic Siegel-domain realization of
$\mathcal D_{p,q}$;
see \cite{kasparian2021toroidal}*{Section 4.4}.
After quotienting by $\Upsilon_P$, the fibers are open
subsets of $T_P(\bbC)$.
The factor $C_P$ in \eqref{eq:horospherical} records the positivity
condition in the Siegel-domain realization of $\mathcal{D}_{p,q}$:
it is the open self-adjoint cone in
$\sqrt{-1}\Lie(U_P)$ appearing in the horospherical coordinates;
see \cite{kasparian2021toroidal}*{Section 4.4}.
In the case of type $I_{p,q}$, one may identify $\Lie(U_P)$ (noncanonically) with the real vector space
of $s\times s$ skew-Hermitian complex matrices  and the cone $C_P$
with the cone of positive definite Hermitian forms. 
In particular, for $s=1$ we have $C_P=\bbR_{>0}$.

Below, we explain the fan decomposition.
Let $\overline{C}_P$ denote the rational closure of the cone $C_P$ in $\sqrt{-1}\Lie(U_P)$.
A (rational polyhedral) fan $\Sigma_P$ is a decomposition of $\overline{C}_P$ into relatively open
rational polyhedral cones, satisfying the usual compatibility conditions.
Given $\Sigma_P$, the lattice $\Upsilon_P$ and the torus $T_P$, one obtains a toric variety $TV(\Sigma_P)$ containing $T_P$ as a dense open subset.
The fan $\Sigma_P$ must be compatible with the $\Gamma$-action and with adjacencies of boundary components.
A collection $\Sigma=\{\Sigma_P\}_P$ of fans, one for each $\Gamma$-rational boundary component, up to conjugacy,
satisfying these compatibilities is called a $\Gamma$-admissible collection.
Admissible collections exist; they are not unique, and can be chosen so that the resulting compactification is
projective.

The quotient $\Upsilon_P\backslash\mathcal D_{p,q}$
is an open subset of a $T_P$-torsor over $\mathcal B_P$.
Given $\Sigma_P$, applying the torus embedding
$T_P\hookrightarrow TV(\Sigma_P)$ relatively over
$\mathcal B_P$ gives the partial toroidal compactification
$(\Upsilon_P\backslash\mathcal D_{p,q})_{\Sigma_P}$.
Since $\Gamma_P$ normalizes $\Upsilon_P$,
the action of $\Gamma_P$ on this partial compactification
factors through $\Gamma_P/\Upsilon_P$.
Taking these quotients gives the toroidal
compactification
$\overline{\mathrm{Sh}^0(\Gamma)}^\Sigma$.
There is a canonical proper morphism
  $\pi_{\Sigma}:\ \overline{\mathrm{Sh}^0(\Gamma)}^{\Sigma}\to \mathrm{Sh}^0(\Gamma)^*$
extending the identity on $\mathrm{Sh}^0(\Gamma)$.

The key structural fact is that every point of $\overline{\mathrm{Sh}^0(\Gamma)}^{\Sigma}$ has a neighborhood contained in a quotient of a toric
variety. More precisely, locally near a boundary stratum associated with $P$,
the space $\overline{\mathrm{Sh}^0(\Gamma)}^{\Sigma}$ is obtained
from a toric chart in $TV(\Sigma_P)$, a neighborhood in the complex
vector space
coming from $\Lie(V_P)$, and a neighborhood in the
smaller boundary symmetric domain $F_P$, by taking a finite quotient \cite{kasparian2021toroidal}*{Section 2.3}. 
Even if $TV(\Sigma_P)$ is smooth (e.g. if $\Sigma_P$ is smooth), the action of the stabilizer of each point may
produce quotient singularities.
We call a $\Gamma$-admissible fan \emph{inversion-free} if,
whenever an element $\gamma$ of $\Gamma_P$ preserves a cone
$\sigma$ setwise, 
$\gamma$ fixes $\sigma$ pointwise.
Equivalently, it fixes the primitive generator of every ray of
$\sigma$.
According to \cites{ash2010smooth,tai1982kodaira}, after replacing $\Sigma$ by a regular inversion-free
$\Gamma$-admissible refinement, every point $x$ 
of $\overline{\mathrm{Sh}^0(\Gamma)}^{\Sigma}$ lying over a rational boundary component that corresponds to $P$ has an analytic neighborhood of the form $(U_{\sigma}\times \mathcal V_P\times \Omega)/G_x$,
where
\begin{itemize}
    \item $U_{\sigma} = \Spec \bbC[\sigma^\vee\cap M_P]$ is a smooth affine toric chart associated with a cone $\sigma\in \Sigma_P$ where $M_P$ is the character lattice of the torus $T_P$;
    \item $\mathcal V_P$ is a sufficiently small open subset of
$\Mat_{p-s,s}(\bbC)\oplus\Mat_{s,q-s}(\bbC)$
in a local holomorphic trivialization of the base over $F_P$;.
    \item $\Omega$ is a sufficiently small open neighborhood in the smaller Hermitian
    symmetric domain $F_P$;
\item $G_x = \Stab_{\Gamma_P/\Upsilon_P}(\widetilde{x})$ for a chosen lift $\widetilde{x}\in U_{\sigma}\times \mathcal V_P\times \Omega$.
    \end{itemize}
    Moreover, if an element of $G_x$ preserves the chart $U_\sigma$, then it preserves each toric boundary divisor attached to a ray of $\sigma$. In particular, after choosing monomial coordinates on $U_\sigma$ and local analytic coordinates on $\mathcal V_P$ and $\Omega$, the singularities in such a neighborhood are finite quotient singularities induced by the image of $G_x$ in
    $\GL\bigl(T_{\tilde{x}}(U_\sigma\times\mathcal V_P\times\Omega)\bigr)$.

Let
\[
  Z_L
  :=
  \ker\bigl(\U(L)\to \Aut(\mathcal{D}_{p,q})\bigr)
  =
  \OO_E^\times\cdot\id_L,
  \qquad
  \Gamma':=\langle\Gamma,Z_L\rangle.
\]
        We next compare the local charts associated with $\Gamma$ and $\Gamma'$, in the case of $\U(p,q)$.
For $\Lambda\in\{\Gamma,\Gamma'\}$, put $\Upsilon_P(\Lambda):=\Lambda\cap U_P$
and $Y_P(\Lambda)
  :=
  \bigl(
    \Upsilon_P(\Lambda)\backslash\mathcal{D}_{p,q}
  \bigr)_{\Sigma_P}$.
After replacing $\Sigma_P$ by a common
regular refinement for the two lattices $\Upsilon_P(\Gamma)\subset\Upsilon_P(\Gamma')$,
we use the same fan $\Sigma_P$ in the definitions of
$Y_P(\Gamma)$ and $Y_P(\Gamma')$.
Let $\varpi_P:Y_P(\Gamma)\to Y_P(\Gamma')$
be the natural finite morphism. 

\begin{prop}
\label{prop:boundary-no-qref}
Fix a rational boundary component associated with $P$ and a ray
$\rho\in\Sigma_P$. Let $D_\rho\subset Y_P(\Gamma)$ and
$D'_\rho\subset Y_P(\Gamma')$ be the corresponding boundary divisors,
and put
\[
  m_\rho
  :=
  \left[
    \mathbb R\rho\cap\Upsilon_P(\Gamma')
    :
    \mathbb R\rho\cap\Upsilon_P(\Gamma)
  \right].
\]
Then the following assertions hold.

\begin{enumerate}
\item
We have $\varpi_P^*D'_\rho=m_\rho D_\rho$.
In particular, $\varpi_P$ is ramified along $D_\rho$ if and only if
$m_\rho>1$. In that case, its inertia group is cyclic of order
$m_\rho$.
\item
For every point $\widetilde{x}\in Y_P(\Gamma')$, the stabilizer of $\widetilde{x}$ does not act as a quasi-reflection.
\end{enumerate}
\end{prop}
\begin{proof}
The statements follow essentially from \cite{ma2022boundary}*{Theorem 1} and the proof therein. 
For completeness, we write the proof below.
Let $\pi:\U(p,q)\to\mathrm{PU}(p,q)$ be the natural projection and
$\overline{\Gamma}:=\pi(\Gamma)=\pi(\Gamma')$.
As explained above, after identifying $U_P$ with its image under
$\pi$, the translation lattice of the arithmetic group
$\overline{\Gamma}$ is precisely $\Upsilon_P(\Gamma')$.

We first prove (1). Since
$\Upsilon_P(\Gamma)\subset\Upsilon_P(\Gamma')$ has finite index, along the ray $\rho$, the inertia group of the morphism
\[
\varpi_P:
(\Upsilon_P(\Gamma)\backslash\mathcal D_{p,q})_{\Sigma_P}
\longrightarrow
(\Upsilon_P(\Gamma')\backslash\mathcal D_{p,q})_{\Sigma_P}
\]
is $(\mathbb R\rho\cap\Upsilon_P(\Gamma'))/(\mathbb R\rho\cap\Upsilon_P(\Gamma))$;
see also \cite{ma2022boundary}*{Lemma 2}.
This group is cyclic of order $m_\rho$.
If $q_\rho$ and $q'_\rho$ are local toric parameters normal to
$D_\rho$ and $D'_\rho$, respectively, then, after multiplication
by a unit, $q'_\rho=q_\rho^{m_\rho}$.
Hence $\varpi_P^*D'_\rho=m_\rho D_\rho$.

We next prove (2).
Consider the natural map $Y_P(\Gamma')
\to
\overline{\Gamma}\backslash\mathcal D_{p,q}^{\Sigma}$.
 By \cite{ma2022boundary}*{Theorem 1}, this
morphism has no boundary divisor as a ramification divisor.
Suppose that an element of the stabilizer of
$\widetilde{x}$ acts as a quasi-reflection on
$T_{\widetilde{x}}Y_P(\Gamma')$.
Its fixed locus then contains a divisor through
$\widetilde{x}$.
The divisor is not contained in the boundary.
The general point of the divisor belongs to
the interior. 
The quotient map would then be ramified in codimension
one in the interior, contradicting Corollary
\ref{cor:branch-codimge2}.
Thus no nontrivial element of the stabilizer acts as a
quasi-reflection.
\end{proof}

As Gritsenko--Hulek--Sankaran \cite{Gritsenko2007kodaira} and Behrens  \cite{behrens2012singularities} implicitly assumed, we use the following lemma to reduce the problem to $\Gamma'$. 
\begin{lem}
\label{lem:irregular singularity to non-irregular}
    If $\overline{\mathrm{Sh}^0(\Gamma')}^\Sigma$ has at worst canonical singularities, then $\overline{\mathrm{Sh}^0(\Gamma)}^\Sigma$ does also.
\end{lem}
\begin{proof}
        The claim follows by the same argument as in \cite{ma2025irregular}*{Proposition 7.4}.
        In fact, for the same rational fan, the two toroidal
quotients are naturally isomorphic. Locally, the finite map
$Y_P(\Gamma)\to Y_P(\Gamma')$ is the quotient by
$\Upsilon_P(\Gamma')/\Upsilon_P(\Gamma)$.
Since $\Gamma'=\Gamma Z_L$ and $Z_L$ acts trivially on the
domain, these local identifications agree
on overlaps and therefore glue.
\end{proof}

By this lemma, in the rest of this section, we may assume that $\Gamma$ contains $Z_L$ for the analysis of the singularities along the boundary.
Henceforth, we assume that $Z_L\subset\Gamma$.

\subsection{Boundary actions}
Let $p,q$ be positive integers and let
$1\leq s\leq\min\{p,q\}$. An $s$-dimensional totally isotropic
$E$-subspace $I\subset L_E$ determines a rational boundary component.
At a toroidal boundary point $\widetilde x=(u_0,v_0,\omega_0)$ before taking quotients, we have
\[
T_{\widetilde x}(U_\sigma\times\mathcal V_P\times\Omega)
\cong T_{u_0}U_\sigma
\oplus\Mat_{p-s,s}(\mathbb C)
\oplus\Mat_{s,q-s}(\mathbb C)
\oplus T_{\omega_0}\mathcal D_{p-s,q-s}.
\]

Choose a weak Witt decomposition
$L\otimes_{\mathcal O_E}E=I\oplus M\oplus I^\vee$
over $E$, and then choose complex bases after fixing
$E\hookrightarrow\mathbb C$, adapted to the isotropic subspace defining the boundary component, with $\dim I=s$ and with $M$ of signature $(p-s,q-s)$ with the corresponding hermitian form $h$. Choosing the Levi subgroup associated with the above Witt
decomposition $L_P\cong \GL_E(I)\times \U(M)$, for an element $\gamma$ of the stabilizer,
let $(a_\gamma,d_\gamma)\in \GL_E(I)\times\U(M)$
denote its Levi component.  
Under a suitable choice of a basis, $\gamma$ is given as a matrix
\begin{equation}\label{description_gamma}
  \gamma =
  \begin{pmatrix}
    a_{\gamma} & -b_{\gamma}^* h d_\gamma & c_{\gamma} (a_{\gamma}^*)^{-1} - 2^{-1}b_\gamma^* h b_\gamma (a_\gamma^*)^{-1}\\
    0 & d_{\gamma} & b_\gamma (a_\gamma^{*})^{-1} \\
    0 & 0 & (a_{\gamma}^*)^{-1}
  \end{pmatrix}\in \U(L)(\bbR),
\end{equation}
where
$a_\gamma\in\GL_s(\bbC)$,
$b_\gamma\in\Mat_{p+q-2s,s}(\bbC)$,
$c_\gamma\in\Mat_s(\bbC)$ with $c_\gamma^*=-c_\gamma$,
and $d_\gamma^*h d_\gamma=h$.
The entries $b_\gamma,c_\gamma$ may affect the eigenvalues in
directions normal to the toric boundary, since translations
in horospherical coordinates become multiplicative factors
in toric coordinates.
Their effects are included in the differential
$D^{\mathrm{tor}}_{\gamma,\widetilde{x}}$ below.
To obtain a lower bound for the Reid--Shepherd-Barron--Tai sum,
we discard only the nonnegative contributions of these normal
directions, as justified in
Lemma \ref{lem:boundary-interior-comparison}.

Let $\gamma$ be an element of the stabilizer of $\tilde{x}$.
Since the local stabilizer in $\Gamma_P/\Upsilon_P$ is finite,
some positive power of $\gamma$ belongs to $\Upsilon_P$.
Consequently, its Levi components $a_\gamma$ and $d_\gamma$
have finite order.
Since $\gamma$ fixes $\omega_0$, we may choose complex bases
of $M$ such that $\omega_0=0$,
$h=\operatorname{diag}(1_{p-s},-1_{q-s})$, and
\[
d_\gamma=\operatorname{diag}(d_{\gamma,1},d_{\gamma,2}),
\qquad
d_{\gamma,1}\in\U(p-s),\quad d_{\gamma,2}\in\U(q-s).
\]
Write $M=M^+\oplus M^-$ for the corresponding positive
and negative subspaces.
In the Grassmannian realization
\eqref{def_realization_D_p,q}, use the holomorphic coordinates
given by the space spanned by the columns of
\[
\begin{pmatrix}
Z & v_-\\
v_+ & \omega\\
0 & 1_{q-s}\\
1_s & 0
\end{pmatrix}
\]
relative to $I\oplus M^+\oplus M^-\oplus I^\vee$.
Here $Z\in\Mat_s(\bbC)$,
$v_+\in\Mat_{p-s,s}(\bbC)$,
$v_-\in\Mat_{s,q-s}(\bbC)$, and
$\omega\in\Mat_{p-s,q-s}(\bbC)$.
Write
$b_\gamma=\binom{b_{\gamma,+}}{b_{\gamma,-}}$.
The following transformation formulas are obtained directly
from \eqref{description_gamma} in the Grassmannian coordinates
of \eqref{def_realization_D_p,q}.
Indeed, applying $\gamma$ to the left, we have
\[
\gamma
\begin{pmatrix}
Z & v_-\\
v_+ & \omega\\
0 & 1_{q-s}\\
1_s & 0
\end{pmatrix}
\begin{pmatrix}
a_\gamma^* & 0\\
-d_{\gamma,2}^*b_{\gamma,-} & d_{\gamma,2}^*
\end{pmatrix}
=
\begin{pmatrix}
Z' & v_-'\\
v_+' & \omega'\\
0 & 1_{q-s}\\
1_s & 0
\end{pmatrix}.
\]
The first two row blocks give
\[
\begin{aligned}
\omega'
 &=d_{\gamma,1}\omega d_{\gamma,2}^*,\\
v_+'
 &=d_{\gamma,1}v_+a_\gamma^*
   +b_{\gamma,+}-\omega'b_{\gamma,-},\\
v_-'
 &=a_\gamma v_-d_{\gamma,2}^*
   -b_{\gamma,+}^*\omega'+b_{\gamma,-}^*.
\end{aligned}
\]
The coordinate is transformed by
\[
Z'
=
a_\gamma Za_\gamma^*
-b_{\gamma,+}^*d_{\gamma,1}v_+a_\gamma^*
+c_\gamma-\frac12 b_\gamma^*h b_\gamma
-v_-'b_{\gamma,-}.
\]
In particular, its linear part in $Z$ is
$Z\mapsto a_\gamma Za_\gamma^*$.
For a unipotent element, at $\omega=0$ the
coordinates transform as $(v_+,v_-)\mapsto
(v_++b_+,\,v_-+b_-^*)$.
Thus the real unipotent parameter
$b=\binom{b_+}{b_-}$ corresponds to the holomorphic
coordinates $(b_+,b_-^*)$ at $\omega=0$.

After the relative torus embedding,
the differential of $\gamma$ at $\widetilde{x}$
therefore has the block form
\begin{equation}\label{equation_action_d_gamma}
d\gamma_{\widetilde x}=
\begin{pmatrix}
D^{\mathrm{tor}}_{\gamma,\widetilde x}&*&*&*\\
0&(v_+\mapsto d_{\gamma,1}v_+a_\gamma^*)&0&*\\
0&0&(v_-\mapsto a_\gamma v_-d_{\gamma,2}^*)&*\\
0&0&0&(\xi\mapsto d_{\gamma,1}\xi d_{\gamma,2}^*)
\end{pmatrix}
\end{equation}
where $D^{\mathrm{tor}}_{\gamma,\tilde{x}}\colon
T_{u_0}U_\sigma\to T_{u_0}U_\sigma$
denotes the differential of the induced action.
By $\mathrm{RT}_{\Mat_s(\bbC)}(a_\gamma)$, we mean the Reid--Shepherd-Barron--Tai sum of the action $X \mapsto a_\gamma X a_\gamma^*$ for $X \in \Mat_s(\bbC)$.

\begin{lem}
\label{lem:boundary-interior-comparison}
    With the notation above, let $\widetilde{x}=(u_0,v_0,\omega_0)
\in U_\sigma\times\mathcal V_P\times\Omega$
be a lift of a boundary point and let $\gamma$ be a stabilizer with the Levi part $\delta$. Then there exists a point $y$ in the interior such that $y$ is fixed by $\delta$.
The Reid--Shepherd-Barron-Tai sum can be evaluated as
\[
\mathrm{RT}_{T_{\widetilde{x}}
(U_\sigma\times\mathcal V_P\times\Omega)}(\gamma)
\geq
\mathrm{RT}_{T_y\calD_{p,q}}(\delta).
\]
If the right-hand side is zero, then $\gamma$ acts trivially on
$T_{\widetilde{x}}
(U_\sigma\times\mathcal V_P\times\Omega)$.
\end{lem}
\begin{proof}
Write $\delta=
\diag\bigl(a_\gamma,d_\gamma,(a_\gamma^*)^{-1}\bigr)$.
Taking a basis of $I \otimes_E \bbC$ and its dual basis of $I \otimes_E \bbC$, $\delta$ acts on $I \otimes_E \bbC$ and $I^\vee \otimes \bbC$ by $a_\gamma$.
The subspaces
\[
J^\pm:=\{(u,\pm u)\mid u\in\bbC^s\}
\subset (I\oplus I^\vee)\otimes_E\bbC
\]
are therefore $\delta$-invariant and are positive and negative
definite, respectively.
Using the decomposition $M=M^+\oplus M^-$ chosen above, put
\[
V_y^+:=J^+\oplus M^+,\qquad
V_y^-:=J^-\oplus M^-.
\]
Then $y:=V_y^-\in\mathcal D_{p,q}$ is fixed by $\delta$.
The decomposition $T_y\mathcal D_{p,q}
\cong\operatorname{Hom}_{\bbC}(V_y^-,V_y^+)$
shows that its four diagonal actions are
\[
X\mapsto a_\gamma Xa_\gamma^*,\qquad
v_+\mapsto d_{\gamma,1}v_+a_\gamma^*,\qquad
v_-\mapsto a_\gamma v_-d_{\gamma,2}^*,\qquad
\xi\mapsto d_{\gamma,1}\xi d_{\gamma,2}^*.
\]
Consequently,
\[
\mathrm{RT}_{T_y\mathcal D_{p,q}}(\delta)
=
\mathrm{RT}_{\Mat_s(\bbC)}(a_\gamma)
+\mathrm{RT}_{T_{v_0}\mathcal V_P}(\gamma)
+\mathrm{RT}_{T_{\omega_0}F_P}(\gamma).
\]

We claim $\mathrm{RT}_{T_{u_0}U_\sigma}(\gamma) \geq \mathrm{RT}_{\Mat_s(\bbC)}(a_\gamma)$.
If the claim holds, we have
\begin{align*}
\mathrm{RT}_{T_y\mathcal D_{p,q}}(\delta)
&=
\mathrm{RT}_{\Mat_s(\bbC)}(a_\gamma)
+\mathrm{RT}_{T_{v_0}\mathcal V_P}(\gamma)
+\mathrm{RT}_{T_{\omega_0}F_P}(\gamma)\\
&\leq 
\mathrm{RT}_{T_{u_0}U_\sigma}(\gamma)
+\mathrm{RT}_{T_{v_0}\mathcal V_P}(\gamma)
+\mathrm{RT}_{T_{\omega_0}F_P}(\gamma)
= 
\mathrm{RT}(\gamma).
\end{align*}
Let $O(\tau)\subset U_\sigma$ be the torus orbit containing $u_0$,
where $\tau\preceq\sigma$ is the corresponding face. As usual,
\[
T_{u_0}O(\tau)
\cong
\Lie(U_P)_{\bbC}/\mathrm{Span}_{\bbC}(\tau).
\]
Since the fan is chosen as inversion-free, $\gamma$
acts trivially on $\mathrm{Span}_{\bbC}(\tau)$. Hence every nontrivial
eigenvalue of the action
$X\mapsto a_\gamma Xa_\gamma^*$ on
$\Lie(U_P)_{\bbC}\cong\Mat_s(\bbC)$ also contributes on
$T_{u_0}O(\tau)$. The remaining eigenvalues of $d_{\gamma_{\widetilde{x}}}$
come from directions normal to $O(\tau)$ and have
nonnegative contribution to the Reid--Shepherd-Barron--Tai sum. This proves the required
inequality. 

Suppose that $\mathrm{RT}_{T_y\calD_{p,q}}(\delta) = 0$.
Then, $a_\gamma Xa_\gamma^*=X$ for every $X\in\Mat_s(\bbC)$. 
Taking
$X=1_s$ gives $a_\gamma^*=a_\gamma^{-1}$, and hence
$a_\gamma X=Xa_\gamma$ for every $X$, so
$a_\gamma=\lambda \cdot 1_s$ for some $|\lambda|=1$.
Similarly, we have
$d_{\gamma,1}=\lambda \cdot 1_{p-s}$ and
$d_{\gamma,2}=\lambda \cdot 1_{q-s}$.
Thus the Levi component of $\gamma$ is scalar.
Since $\gamma\in\U(L)$, we have $\lambda\in\OO_E^\times$, so $z:=\lambda \cdot \id_L\in Z_L\subset\Gamma$. 
By replacing $\gamma$ by $z^{-1}\gamma$, we have $\gamma\in N_P$. 
Its image in $V_P=N_P/U_P$ acts by translation on
the $\mathcal V_P$-coordinate; since it fixes $\widetilde{x}$, this translation is zero. Therefore $\gamma\in U_P\cap\Gamma=\Upsilon_P$, and in particular $d\gamma_{\widetilde{x}}=\id$.
\end{proof}

We now prove the canonicity of the singularities on the boundaries.

\begin{thm}\label{thm:nonterminal}
Let $\Sigma$ be a regular inversion-free $\Gamma$-admissible fan as above.
Assume $p\geq w_E$ and $(p,q)\neq(2,2), (2,3)$.
Then a toroidal compactification $\overline{\mathrm{Sh}^0(\Gamma)}^{\Sigma}$ has at worst canonical singularities along its boundary.
\end{thm}
\begin{proof}
By Lemma \ref{lem:irregular singularity to non-irregular},
we may assume $Z_L\subset\Gamma$.
Let $\gamma$ be an element with nontrivial action on the tangent space at a boundary point associated with an isotropic subspace of dimension $s$.
Proposition \ref{prop:boundary-no-qref} (2) excludes
quasi-reflections, so it suffices to prove
$\mathrm{RT}(\gamma)\geq1$.
This inequality follows from Lemma \ref{lem:boundary-interior-comparison} and Theorem \ref{thm_classif_sing_disc<-4}.
\end{proof}

\section{Arthur's multiplicity formula and archimedean \texorpdfstring{$A$}{}-packets for \texorpdfstring{$\U(p,q)$}{}}
\label{sec:Archimedean packets in the present case}
In this section, we recall the representation-theoretic input needed to prove Theorem \ref{mainthm:general type}.
We work under the assumption $N=p+q\equiv 1 \pmod{w_E}$ from Theorem \ref{mainthm:general type}; in particular, $N$ is odd.
We follow the notation in \cite{2025_HMY_Kodaira_dimension_ball_quotients}*{Section 3-7}; see also \cite{KMSW}.

\subsection{Arthur's multiplicity formula}

In this subsection, we recall the multiplicity formula in the generic
case used below, where all $d_i=1$. \cites{Mok_2015, KMSW}.
Throughout Sections \ref{sec:Archimedean packets in the present case}, \ref{sec:The Kodaira dimension}, $v$ denotes a place of $\bbQ$, and we write $E_v\defeq E\otimes_{\bbQ}\bbQ_v$.

Let $V$ be an $N$-dimensional Hermitian space over $E$.
Arthur's endoscopic classification gives a decomposition of the discrete automorphic spectrum into near-equivalence classes
\[
L^2_\disc(\U(V)(\bbQ)\backslash \U(V)(\bbA)) \cong \bigoplus_\psi L^2_\psi(\U(V))
\]
where $\psi$ runs over all equivalence classes of global $A$-parameters, that is a formal unordered sum
\[
\psi = \bigoplus_i \pi_i \boxtimes S_{d_i}
\]
such that
\begin{itemize}
    \item $\pi_i$ is an irreducible conjugate-selfdual cuspidal automorphic representation of $\GL_{N_i}(\bbA_E)$ with sign $(-1)^{N-d_i}$;

    \item $S_{d_i}$ is the unique $d_i$-dimensional irreducible algebraic representation of $\SL_2(\bbC)$;
    \item $(\pi_i, d_i) \neq (\pi_j, d_j)$ if $i \neq j$;
    \item $\sum_i N_id_i = N$.
\end{itemize}

Let $\calS_\psi$ be the global component group of $\psi$, that is formally defined by a free $\bbZ/2\bbZ$-module
\[
\calS_{\psi} = \bigoplus_i (\bbZ/2\bbZ)e_i,
\]
where $e_i$ is the symbol corresponding to $\pi_i \boxtimes S_{d_i}$.
The global $A$-parameter gives rise to a character $\vep_\psi$ of $\calS_\psi$ as in \cite{Mok_2015}*{(2.5.5)}.
Note that if $d_i = 1$ for all $i$, then $\vep_\psi$ is trivial.
For a place $v$ of $\bbQ$, the localization $\psi_v$ of $\psi$ at $v$ is an $A$-parameter for $\U_N$ and provides a finite set of semisimple representations of $\U(V_v)$.
We fix a global Whittaker datum \cite{KMSW}*{Subsection 0.2.2} and the split Hermitian form $h_{E/\bbQ}$ defined in \cite{2025_HMY_Kodaira_dimension_ball_quotients}*{Subsection 3.1}.
The localization of the Whittaker datum at $v$ yields a map
\[
\iota_v \colon \Pi(\psi_v) \rightarrow \Irr(\calS_{\psi_v}),
\]
where $\calS_{\psi_v}$ is the component group of the $A$-parameter $\psi_v$.
Let $z_v$ be the central element in $\calS_{\psi_v}$.
At a finite nonsplit place, the central value
$\eta_v(z_v)$ records the invariant of the Hermitian
space relative to the chosen reference form.
At the real place, our normalization gives
$\eta_\infty(z_\infty)=(-1)^q$ for $\U(p,q)$.
For $\eta_v \in \Irr(\calS_{\psi_v})$, we denote by $\pi_v(\psi_v, \eta_v)$ the representation of $\U(V_v)$ corresponding to the character $\eta_v$.
For convenience, if $\eta_v$ does not lie in the image of $\iota_v$, set $\pi_v(\psi_v,\eta_v)=0$.

There exists a natural homomorphism $\Delta_v \colon \calS_{\psi} \rightarrow \calS_{\psi_v}$ (see \cite{KMSW}*{Subsection 1.3.5} and \cite{2025_HMY_Kodaira_dimension_ball_quotients}*{(5.5)}).
For any $(\eta_v)_v \in \prod_v\Irr(\calS_{\psi_v})$, we obtain a representation
\[
\pi(\psi,\eta) = \bigotimes_v \pi_v(\psi_v, \eta_v)
\]
of $\U(V)(\bbA)$.
\begin{thm}[Arthur's multiplicity formula {\cites{Mok_2015,KMSW}}]\label{AMF}
With the notation above, one has
\[
L^2_\psi(\U(V)) \cong \bigoplus_\eta \pi(\psi,\eta),
\]
where $\eta = (\eta_v)_v$ runs over all characters of $\prod_v\calS_{\psi_v}$ such that $\eta_v = \mathbf{1}$ for almost all $v$ and $\prod_v \eta_v \circ \Delta_v = \vep_\psi$.
\end{thm}
\subsection{Archimedean places}

By the highest weight theory, 
irreducible representations of the maximal compact subgroup $\U(p) \times \U(q)$ of $\U(p,q)$ are classified by highest weights
\[
(\lambda_1,\ldots,\lambda_p,\lambda_{p+1},\ldots,\lambda_{N}) \in \bbZ^{N}, \qquad \lambda_1 \geq \cdots \geq \lambda_p,\ \lambda_{p+1} \geq \cdots \geq \lambda_{N}.
\]
For each highest weight $\lambda = (\lambda_1,\ldots,\lambda_p,\lambda_{p+1},\ldots,\lambda_{N})$, we obtain the unique irreducible lowest-weight representation $L(\lambda)$ of $\U(p,q)$ with lowest $K$-type $\lambda$.

The module $L(\lambda)$ is unitarizable if $\lambda_p - \lambda_{p+1} \geq N -p' - q'$ where $p' \defeq \#\{i \mid \lambda_i = \lambda_p, 1 \leq i \leq p\}$ and $q' \defeq \#\{i \mid \lambda_i = \lambda_{p+1}, p+1 \leq i \leq N\}$.
Note that $L(\lambda)$ is a discrete series (resp. limit of discrete series) representation if $\lambda_{p} - \lambda_{p+1} > N-1$ (resp. $\lambda_p - \lambda_{p+1} = N-1$) \cite{EHW_83}*{Theorem 2.4, Proposition 3.1}.
By (limit of) holomorphic discrete series representations, we mean a representation that is both (limit of) discrete series and of lowest weight.
The following statement is a special case of \cite{Horinaga_LW_unitary}*{Theorem 1.2}.
For $t \in \bbZ$, let $\chi_t$ be the character of $\bbC^\times$ defined by $\chi_t(z) = (z/\overline{z})^{t/2} = z^t/(z\overline{z})^{t/2}$.

\begin{lem}\label{A_packet_hol_LDS}
    Let $\psi = \bigoplus_{i=1}^N \chi_{t_i} \otimes S_{1}$ be an $A$-parameter with $t_1\geq \cdots \geq t_N$.
    Let $\lambda = (\lambda_1, \ldots,\lambda_{N})$ be a weight such that $\lambda_1 = \cdots = \lambda_p > \lambda_{p+1} = \cdots = \lambda_{N}$ and $\lambda_p - \lambda_{p+1} = N-1$.
    \begin{enumerate}
        \item The $A$-packet $\Pi(\psi)$ contains $L(\lambda)$ if and only if the $A$-parameter $\psi$ satisfies
            \begin{align*}
        &\bigsqcup_{i=1}^{N}\left\{\frac{t_i}{2}\right\} = \left\{\lambda_1 + \frac{p-q-1}{2}, \ldots, \lambda_{p} - \frac{N-1}{2}\right\} \bigsqcup
        \left\{\lambda_{p+1}+\frac{N-1}{2}, \ldots,\lambda_{N}+\frac{p-q+1}{2}
        \right\}.
        \end{align*}
        \item If $\Pi(\psi)$ contains $L(\lambda)$, then the character $\eta$ corresponding to $L(\lambda)$ is given by 
        \[
        \eta(e_i) = 
        \begin{cases}
            (-1)^{i + (N+1)/2} & \text{if}\ i \leq p;\\
            (-1)^{i+ (N-1)/2} & \text{if}\ p+1 \leq i \leq N.
        \end{cases}
        \]
    \end{enumerate}
\end{lem}

\section{The Kodaira dimension}
\label{sec:The Kodaira dimension}
\subsection{Existence of low weight cusp forms}
We will show the existence of low weight cusp forms using Arthur's multiplicity formula.
As we have seen Theorem \ref{AMF} and Lemma \ref{A_packet_hol_LDS}, our task is to find a global $A$-parameter with the global sign condition $\vep_\psi = \prod_{v} \eta_v \circ \Delta_v$.
It is purely combinatorial under the existence of certain cuspidal representations \cite{2025_HMY_Kodaira_dimension_ball_quotients}*{Theorem A.1}.
We consider the combinatorial counterpart to construct global $A$-parameters with desired properties.
Let $B(E)$ be the integer defined in \eqref{def_B(E)}.
\begin{lem}\label{lem_partition}
    Let $\eta$ be the character as in Lemma \ref{A_packet_hol_LDS} (2).
    If $p+q\equiv 1 \pmod{w_E}$ and $p+q \geq B(E)$, there exists a partition $A\sqcup \bigsqcup_{j \in J} B_j$ of the multiset 
    \[
    \{(3p+q-3)/2,\ldots,(p+q-1)/2\} \sqcup \{(p+q-1)/2,\ldots,(p-q+1)/2\}
    \]
    satisfying the following properties:
    Let $e_i$ be the symbol corresponding to 
    \[
    \begin{cases}
        (3p+q-1)/2-i & \text{if $1 \leq i \leq p$;}\\
        (3p+q-1)/2-i + 1 & \text{if $p<i \leq p+q$,}
    \end{cases}
    \]
    in the multiset.
    Let $e(a),e(b)$ be the symbol corresponding to $a \in A$ and $b \in B_j$, respectively.
    Then, the partition satisfies
    \begin{itemize}
        \item $A$ is singleton such that $\eta(e(a)) = (-1)^q$ and $a$ is a multiple of $w_E$ for $a \in A$;
        \item $\#B_j = 2$ for any $j \in J$;
        \item for $\{k,\ell\} = B_i$, $\eta(e(k)) = \eta(e(\ell))$;
        \item for $\{k,\ell\} = B_i$, $k+\ell$ is a multiple of $w_E$ with
        \[
        |k-\ell| \geq \begin{cases}
            8 & \text{if $D = 3$;}\\
            6 & \text{if $D=4, 7$;}\\
            4 & \text{if $D=8,11,15,23$;}\\
            2 & \text{if $D \geq 19$ and $D \neq 23$.}
        \end{cases}
        \]
    \end{itemize}
\end{lem}
\begin{proof}
    We first reduce the statement to the cases \[
B(E)\leq p+q\leq2B(E)-2,\qquad
p+q\equiv1\pmod{w_E},\qquad p,q\geq1.
\]
    Suppose that for a given $(p,q)$, there exists a partition $A \sqcup \bigsqcup_j B_j$.
    We construct a partition for $(p+B(E)-1, q)$.
    The case $(p,q+B(E)-1)$ is similar.
    Consider the set
    \begin{align*}
    \left\{\frac{3p+q-6}{2}+\frac{3B(E)}{2},\ldots,\frac{p+q-2}{2}+\frac{B(E)}{2}\right\} \bigsqcup \left\{\frac{p+q-2}{2}+\frac{B(E)}{2},\ldots,\frac{p-q}{2}+\frac{B(E)}{2}\right\}.
    \end{align*}
    Let $A'$ and $B_{j+(B(E)-1)/2}'$ be the sets obtained by adding $(B(E)-1)/2$ to $A$ and $B_j$.
    The remaining set is
\[\left\{\frac{3p+q-6}{2} + \frac{3B(E)}{2},\ldots,\frac{3p+q-2}{2}+\frac{B(E)}{2}\right\}\].
    By the definition of $B(E)-1$, we may find a partition $B_1' \sqcup \cdots \sqcup B_{(B(E)-1)/2}'$ such that the resulting partition $A' \sqcup \bigsqcup_j B_j'$ satisfies the assumptions in the lemma.

    When $\max\{p,q\} \leq B(E)-1$ and $E \neq \bbQ(\sqrt{-1}), \bbQ(\sqrt{-3})$, one can construct partitions based on \cite{2025_HMY_Kodaira_dimension_ball_quotients}*{Table 2}.
    For the cases $E = \bbQ(\sqrt{-1}), \bbQ(\sqrt{-3})$, the statement follows from the explicit constructions of the partition $A \sqcup \bigsqcup_j B_j$ for finitely many cases $\max\{p,q\} \leq B(E)-1$.  
\end{proof}

As an application of Lemma \ref{lem_partition}, we show the existence of low weight cusp forms.

\begin{thm}
\label{thm:low weight cusp form}
Assume that $p+q\equiv 1 \pmod{w_E}$ and $p+q \geq B(E)$.
Then, there exists a nonzero cusp form of weight $p+q-1$ with respect to $\U(L)$.
\end{thm}

\begin{proof}
Set $N = p+q$.
By \cite{2025_HMY_Kodaira_dimension_ball_quotients}*{Theorem A.1}, even in the case $D$ even, there exists a normalized Hecke eigenform $f(z) = \sum_{n} a_n \exp(2\pi \sqrt{-1}\ nz) \in S_k(D, \omega_{E/\bbQ})$ such that $a_p \neq \pm p^{(k-1)/2}$ for any $p \mid D$ if 
\[
    k \geq 
    \begin{cases}
            8 & \text{if $D = 3$;}\\
            6 & \text{if $D=4, 7$;}\\
            4 & \text{if $D=8,11,15,23$;}\\
            2 & \text{if $D \geq 19$ and $D \neq 23$.}
    \end{cases}
\]
We denote by $\pi_k = \bigotimes'_v \pi_{k,v}$ the cuspidal representation of $\GL_2(\bbA)$ generated by such $f$.
The base change lift $\pi_k^\BC$ to $\GL_2(\bbA_E)$ is still cuspidal by \cite{2025_HMY_Kodaira_dimension_ball_quotients}*{Theorem A.1} (see also \cite{2025_HMY_Kodaira_dimension_ball_quotients}*{Section 7.1}).
By \cite{2012_Loeffer-Weinstein}*{Proposition 2.8}, the representation $\pi_{k,p}$ of $\GL_2(\bbQ_p)$ is a principal series representation induced from $\mu_p \boxtimes \mu_p^{-1}\omega_{E/\bbQ,p}$ so that $\mu_p$ is unramified with $\mu_p(p) = a_p(f)/p^{(k-1)/2} \in S^1$.
The base change lift $\pi_{k,p}^\BC$ has the $L$-parameter $\mu_p|_{L_{E_v}} \oplus \mu_p^{-1}|_{L_{E_v}}$, where $L_{E_v}$ is the Weil-Deligne group of $E_v$.
Hence, the $L$-parameter of $\pi_{k,p}^\BC$ is unramified even for $p=2$.
For the archimedean place $\infty$, the $L$-parameter of $\pi_{k,\infty}^\BC$ is given as $\chi_{k-1} \oplus \chi_{1-k}$.

We now construct a global $A$-parameter.
Suppose that $N\equiv 1 \pmod{w_E}$ and $N \geq B(E)$.
Take a partition $A \sqcup \bigsqcup_i B_i$ in Lemma \ref{lem_partition}, and put $A = \{a\}, B_i = \{k_i, \ell_i\}$ with $k_i<\ell_i$.
We then define a global $A$-parameter by
\[
\psi = 
\chi^{2a/w_E} \boxtimes S_1  \oplus \bigoplus_i (\pi^{\BC}_{\ell_i-k_i+1} \otimes \chi^{(k_i+\ell_i)/w_E}) \boxtimes S_1.
\]
Here, $\chi$ is a character of $E^\times \backslash\bbA_E^\times$ such that
\[
\chi_v|_{\OO_{E_v}^\times} = \mathbf{1}, \qquad \chi_\infty(z) = \chi_{w_E}(z) = (z/\overline{z})^{w_E/2}
\]
and $\pi^\BC$ is the base change lift of $\pi$ to $\GL_2(\bbA_E)$.
The archimedean component $\psi_\infty$ of $\psi$ is given by
\begin{align}
\psi_\infty &= 
\chi^{2a/w_E}_{w_E} \boxtimes S_1 \oplus \bigoplus_i ((\chi_{k_i-\ell_i} \oplus \chi_{\ell_i-k_i})\otimes \chi_{w_E}^{(k_i+\ell_i)/w_E})\boxtimes S_1 \notag\\
&= 
\chi_{2a} \boxtimes S_1 \oplus \bigoplus_i (\chi_{2k_i} \oplus \chi_{2\ell_i})\boxtimes S_1.
\label{description_psi_infty}
\end{align}
From $\psi_\infty$, we may associate symbols $e_{1,\infty}, \ldots, e_{N, \infty}$ to each representation in $\psi_\infty$ so that the corresponding subscript of $\chi$ in \eqref{description_psi_infty} is weakly decreasing when the subscripts $i$ increase.
By construction, the representation $\pi(\psi_\infty,\eta) \in \Pi_{\psi_\infty}$ is the weight $N-1$ lowest weight representation of $\U(p,q)$.
In fact, the multiset of half the subscripts of $\chi$ in \eqref{description_psi_infty} is equal to the multiset $A \sqcup \bigsqcup_j B_j$.
Hence, the corresponding multiset $\bigsqcup_i \{t_i/2\}$ in Lemma \ref{A_packet_hol_LDS} is the same as $A \sqcup \bigsqcup_j B_j$.
The multiset $A \sqcup \bigsqcup_j B_j$ is the same as that for the limit of the holomorphic discrete series representation with $\lambda_1 = \cdots = \lambda_p = N-1$ and $\lambda_{p+1} = \cdots = \lambda_N = 0$.

Let $\eta$ be the character in Lemma \ref{A_packet_hol_LDS}.
Define the symbol $e_{i,\bbA}$ for $i=1,2,\ldots$ to a representation in $\psi$.
If necessary, we reorder the indices $i$ of $e_{i,\bbA}$ so that the maximal subscript of $\Delta_\infty(e_{i,\bbA})$ decreases when $i$ increases.
The component group for $\psi$ is
\[
\bigoplus_i(\bbZ/2\bbZ) e_i.
\]
Let $i_0$ be the index of $e_{i_0,\bbA}$ corresponding to $\chi^{2a/w_E} \boxtimes S_1$.
By definition, we have
\[
\eta \circ \Delta_\infty(e_{i_0, \bbA}) = (-1)^q = \eta_\infty(z_\infty), \qquad \eta \circ \Delta_\infty(e_{i,\bbA}) = 1
\]
for any $i \neq i_0$.
For a finite place $p$ with $p \nmid D$, the localization $\psi_p$ provides an unramified $A$-parameter.
If $p$ splits in $E$, since $\U(L)(\bbQ_p)$ is isomorphic to $\GL_{N}(\bbQ_p)$, the component group $\calS_{\psi_p}$ is $\{1\}$ and so the corresponding character $\eta_p$ is always trivial.
If $p$ is an unramified nonsplit, the component group $\calS_{\psi_p}$ is $\{1, z_{p}\}$ and we choose the character $\eta_v$ so that $\eta_v(z_{p})$ is $+1$ (resp. $-1$) if the Hermitian form $L_p$ is isomorphic to $h_{E_v/\bbQ_p}(N)$ (resp.\ if not).
For $p \mid D$, 
by the same discussion as in \cite{2025_HMY_Kodaira_dimension_ball_quotients}*{Lemma 7.7}, the $A$-parameter $\psi_p$ is also unramified and $\calS_{\psi_p} = \{1,z_{\psi_p}\}$.
By construction, the character $\eta_v \circ \Delta_v$ is given by
\[
\eta_v \circ \Delta_v(e_{i,\bbA})
=
\begin{cases}
    1 & \text{if $i \neq i_0$;}\\
    \eta_v(z_v) & \text{if $i = i_0$.}
\end{cases}
\]
By the product formula $\prod_v \eta_v(z_v) = 1$, we have $\eta \circ \Delta = \mathbf{1} = \vep_\psi$.
Hence, the representation $\pi(\psi, \eta)$ is automorphic and cuspidal by Theorem \ref{AMF} and the genericity of $A$-parameter $\psi$ (cf. \cite{1984_Wallach_constant_term}).

Let $\pi = \otimes_v' \pi_v$ be the automorphic representation $\pi(\psi,\eta)$.
For any finite place $v$, the representation $\pi_v$ is unramified.
By the construction of $\psi$ and the proof of \cite{2025_HMY_Kodaira_dimension_ball_quotients}*{Lemma 7.7}, the space $\pi^K$ is nonzero for any open compact subgroup $K$ of $\U(L)(\bbA_\fini)$.
The isomorphism between automorphic forms and modular forms in \cite{2025_HMY_Kodaira_dimension_ball_quotients}*{Lemma 4.3} yields a nonzero cusp form of weight $N-1$ (see also \cite{2025_HMY_Kodaira_dimension_ball_quotients}*{Subsection 7.2}).
\end{proof}

\subsection{An application to the Kodaira dimension}
We keep the notation in Proposition \ref{prop:boundary-no-qref}.
Note that the ray $\rho$ is regular for $\Gamma$ precisely when $m_\rho=1$; otherwise it is an
irregular ray for $\Gamma$.

\begin{proof}[Proof of Theorem \ref{mainthm:general type}]
Set $N=p+q$.  By Theorem \ref{thm:low weight cusp form}, there exists a
nonzero cusp form $F$ of weight $N-1$ with respect to $\U(L)$. 
We first assume that $Z_L\subset\Gamma$.  In this case $\Gamma=\Gamma'$.
After passing to 
$\mathrm{PU}(p,q)$, Proposition \ref{prop:boundary-no-qref} (2) shows that 
there is no boundary divisor as a ramification divisor.
Together with
Corollary \ref{cor:branch-codimge2}, Hirzebruch--Mumford proportionality \cite{mumford1977hirzebruch}
gives
\begin{equation}\label{eq:canonical-center-contained}
K_{\overline{\mathrm{Sh}^0(\Gamma)}^{\Sigma}}
 \sim_{\bbQ}N\L-\Delta,
\end{equation}
where $\L$ denotes the automorphic $\bbQ$-line bundle.
Since $F$ is cuspidal, its extension to the toroidal compactification
vanishes along every irreducible component of the boundary
$\Delta$ with order at least one.  Hence $(N-1)\L-\Delta$ is effective.

We next consider the case $Z_L\not\subset\Gamma$.  
Since $F$ is a cusp form for $\Gamma'\subset\U(L)$, pulling back along the natural finite map and using Proposition \ref{prop:boundary-no-qref} (1), we obtain, for every ray $\rho$,
$\ord_{D_\rho}(\varpi_P^*F)=m_\rho\ord_{D_\rho'}(F)\geq m_\rho$.
Thus, on the cover,
\begin{align*}
\div(\varpi_P^*F)-\sum_\rho D_\rho
&=\left(\div(\varpi_P^*F)-\sum_\rho m_\rho D_\rho\right)
  +\sum_\rho(m_\rho-1)D_\rho
\geq0.
\end{align*}
These local identities descend to the whole toroidal compactification, and hence $(N-1)\L-\Delta$ is effective.

Finally, $\L$ is big by \cite{Baily1966compactification} 
and
Hirzebruch--Mumford proportionality \cite{mumford1977hirzebruch}.  Since the canonical bundle can be written as the sum of an effective bundle and a big bundle
\[
K_{\overline{\mathrm{Sh}^0(\Gamma)}^{\Sigma}}
 \sim_{\bbQ}\bigl((N-1)\L-\Delta\bigr)+\L,
\]
it is also big.  By Theorem \ref{mainthm:singularity}, the fan $\Sigma$ may
be chosen so that the toroidal compactification has at worst canonical
singularities.  Hence pluricanonical forms extend to a resolution, and
$\overline{\mathrm{Sh}^0(\Gamma)}^{\Sigma}$ is of general type.
\end{proof}

\section{Remarks on the singularities arising from other group actions}
\label{sec:other_groups}
The purpose of this section is to improve the results on singularities of orthogonal modular varieties \cite{Gritsenko2007kodaira} or ball quotients \cite{behrens2012singularities}.
As in Lemma \ref{lem:irregular singularity to non-irregular}, for each theorem it suffices to consider the case in which $\Gamma$ contains the center.
In this section we retain the notation $\mathrm{RT}$ for the Reid--Shepherd-Barron--Tai sums for the case of orthogonal modular varieties or ball quotients.
For simplicity, we will omit the notation on $\Sigma$.
\subsection{The case of \texorpdfstring{$\O^+(2,n)$}{}}

We follow the method of \cite{Gritsenko2007kodaira} with more detailed computations.
As a consequence of Theorems \ref{thm_GHS_interior}, \ref{thm_orth_interior}, \ref{thm_GHS_0_dim_cusp} and \ref{thm_oethogonal_dim_1_cusp}, we refine \cite{Gritsenko2007kodaira}*{Theorem 2}.

\begin{thm}
\label{thm:orthogonal_refinement}
    Let $L$ be a lattice of signature $(2,n)$ with $n \geq 6$, and let $\Gamma < \mathrm{O}^+(L)$ be a subgroup of finite index.
        Then there exists a toroidal compactification $\overline{\calF_L(\Gamma)}$ of $\calF_L(\Gamma)=\Gamma\backslash\mathcal{D}_L$ such that $\overline{\calF_L(\Gamma)}$ has canonical singularities.
\end{thm}

We also show that the bound $n\geq6$ is sharp: there exist $L$ and $\Gamma$ such that $\overline{\calF_L(\Gamma)}$ has noncanonical singularities for $2\leq n \leq 5$.

\subsubsection{Interior}
Let $L$ be a lattice of signature $(2,n)$ and $\mathcal{D}_L$ be the associated Hermitian symmetric space.
For a congruence subgroup $\Gamma$ of $\O^+(2,n)$ with $-1\in\Gamma$, Gritsenko--Hulek--Sankaran proved a criterion for canonicity of $\calF_L(\Gamma)\defeq\Gamma\backslash\mathcal{D}_L$.
For $x\in\mathcal D_L$, let
$\Gamma_x:=\operatorname{Stab}_\Gamma(x)$, and denote its image
in $\mathcal F_L(\Gamma)$ by $[x]$.
\begin{thm}[{\cite{Gritsenko2007kodaira}*{Theorem 2.10, Proposition 2.15}}]\label{thm_GHS_interior}
    \begin{enumerate}
                \item[(1)] If no nontrivial $g\in\Gamma_x$ acts as a quasi-reflection and $n\geq6$, then $\calF_L(\Gamma)$ has at worst a canonical singularity at $[x]$.
        \item[(2)] If some $g \in \Gamma_x$ acts as a quasi-reflection and $n \geq 7$, then $\calF_L(\Gamma)$ has at worst a canonical singularity at $[x]$.
    \end{enumerate}
\end{thm}
Item (1) of the theorem above is sharp.
In this subsection, we refine the second case as follows:
\begin{thm}\label{thm_orth_interior}
    If some $g \in \Gamma_x$ acts as a quasi-reflection and $n \geq 6$, then $\calF_L(\Gamma)$ has at worst a canonical singularity at $[x]$.
\end{thm}
We then have a refinement of \cite{Gritsenko2007kodaira}*{Corollary 2.16}.
\begin{cor}\label{cor_orthogonal_interior}
    The orthogonal modular variety $\calF_L(\Gamma)$ has at worst canonical singularities if $n \geq 6$.
\end{cor}

Recall the notation as in \cite{Gritsenko2007kodaira}.
In this subsection, we assume $n>2$ for simplicity.
Fix $x \in \mathcal{D}_L$ and consider the stabilizer $\Gamma_x$.
For $g \in \Gamma_x$, the $g$-module $L_\bbQ = L \otimes \bbQ$ decomposes as
\[
\calV_{n_0} \oplus \bigoplus_{j \geq 1} \calV_{n_j}
\]
where $\calV_n = W_{\Phi_n}$.
Let $\{a_{i,j} \mid i = 1,\ldots,\varphi(n_j)\} \subset \bbQ/\bbZ$ be the set of exponents of the eigenvalues of $g$ on $\calV_{n_j}$.
Then, there exists $i_0$ such that the multiset of exponents of the action of $g$ on the tangent space $T_x \mathcal{D}_L$ is given by
\[
\{a_{i,0} + a_{i_0,0} \mid i \neq i_0,i_0'\} \sqcup \bigsqcup_{j \geq 1}\{a_{i, j} + a_{i_0,0} \mid i = 1,\ldots,\varphi(n_j)\}
\]
where $i_0'$ is the unique integer such that $a_{i_0',0}=-a_{i_0,0}$.
Without loss of generality, we may assume $i_0 = 1, i_0'=2$.
Suppose some power of $g$ acts as a quasi-reflection on the tangent space $T_x\mathcal{D}_L$.
Let $k$ be the least positive integer such that the action of $g^k$ is a quasi-reflection.
\begin{thm}[{\cite{Gritsenko2007kodaira}*{Theorem 2.12}}]
\label{thm:GHS_2.12}
    With the above notation, as a $g$-module, the decomposition $L_\bbQ = \calV_{n_0} \oplus \bigoplus_{j>0} \calV_{n_j}$ satisfies either of the following:
    \begin{itemize}
        \item $(n_0,k)=n_0$ and $2(n_j,k)=n_j$ for some $j>0$;
        \item $2(n_0,k)=n_0$ and $(n_j,k)=n_j$ for some $j>0$.
    \end{itemize}
\end{thm}
We need the following refinement of the preceding theorem.
\begin{prop}\label{prop_decomp_L_quasi_ref}
    With the above notation, as a $g$-module, the decomposition $L_\bbQ = \calV_{n_0} \oplus \bigoplus_{j>0} \calV_{n_j}$ satisfies either of the following:
    \begin{description}
        \item[Case 1] There exists $j_0 > 0$ such that $k \in 2\bbZ+1, n_{j_0} = 2$ and $(n_{j}, k) = n_{j}$ for any $j \neq j_0$;
        \item[Case 2]  There exists $j_0 > 0$ such that $n_{j_0} = 1$ and $2(n_{j}, k) = n_{j}$ for any $j \neq j_0$.
    \end{description}
\end{prop}
\begin{proof}
    In the first case of Theorem \ref{thm:GHS_2.12}, since $n_0$ divides $k$, the action of $g^k$ on $\calV_{n_0}$ is trivial.
    The action of $g^k$ on $\calV_{n_j}$ for $j > 0$ is trivial except for a unique $j = j_0$.
    Hence, for $j \neq j_0$, we have $n_j \mid k$.
    Since the action of $g^k$ on $\calV_{n_0}$ is trivial, $g^k$ acts as a quasi-reflection on $\calV_{n_{j_0}}$.
    We then have $\dim_\bbQ \calV_{n_{j_0}} = \varphi(n_{j_0}) = 1$ which implies that $n_{j_0} = 2$.
    This is exactly the case 1.
    The second case is analogous, since the action of $g^k$ on $\calV_{n_0}$ is scalar $-1$.
\end{proof}

Since $g^k$ acts as a quasi-reflection, we need to consider another criterion for the canonical singularities.
Let $\zeta_{2k}^{a_1},\ldots,\zeta_{2k}^{a_n}$ be the eigenvalues of $g$ on $T_x\mathcal{D}_L$ such that $0 \leq a_i < 2k$ for any $i$.
Since $g^k$ acts as a quasi-reflection, we have $ka_n/2k = 1/2$ and $ka_i/2k = 0$ in $\bbQ/\bbZ$ for any $i=1,\ldots,n-1$.
We define
\begin{equation}\label{def_sigma'}
\mathrm{RT}'(g^\ell) \defeq \left\{\frac{\ell a_n}{k}\right\} + \sum_{i=1}^{n-1} \left\{\frac{\ell a_i}{2k}\right\}. 
\end{equation}
Consider the character $\alpha$ of the cyclic group generated by $g$ defined by
$\alpha(g) = e^{2\pi i a_{2,0}}$
as defined in \cite{Gritsenko2007kodaira}*{p.525}.
We now prove Theorem \ref{thm_orth_interior}.

\begin{proof}[Proof of Theorem \ref{thm_orth_interior}]
    Suppose that $g^k$ acts as a quasi-reflection and $k$ is the least positive integer such that $g^k$ does so.
    We may assume that $k >1$ and $j_0 = 1$ as in Proposition \ref{prop_decomp_L_quasi_ref}.
    By \cite{Gritsenko2007kodaira}*{Lemma 2.14}, it suffices to show $\mathrm{RT}'(g^\ell) \geq 1$ for any $1 \leq \ell <k$ under the assumption that $n \geq 6$. 
    
        We argue by induction on the order of $\alpha$.
    If the order of $\alpha$ is $1$, this is the first case as in Proposition \ref{prop_decomp_L_quasi_ref}.
    If there is no $j$ such that $\varphi(n_j) \geq 2$, then $n_j = 1$ or $2$ for any $j$.
    Since $g^k$ acts as a quasi-reflection, we have $n_j = 1$ for $j \neq j_0$ and hence $k=1$, which contradicts $k>1$.
    If there is some $j=j'$ such that $\varphi(n_{j'}) \geq 2$, by definition, $\mathrm{RT}'(g)$ has the lower bound
    \[
    \mathrm{RT}'(g) \geq \mathrm{RT}_{\calV_{n_j}}(g) \geq s_{\min}(\Phi_{n_j}, \bbQ) \geq 1
    \]
    where we define $s_{\min}(f,\bbQ)$ in a similar way as $s_{\min}(f,E)$.
        We may apply the same argument to every $g^\ell$, since the character $\alpha$ associated with $g^\ell$ is trivial.
    If $\alpha$ has order $2$, the same argument proves the claim.

        Suppose that $\alpha$ has order $r\geq3$.
        For $\ell$ with $(\ell,r)\neq1$, the character corresponding to $g^\ell$ has order strictly less than $r$.
    By the inductive assumption, we have $\mathrm{RT}'(g^\ell) \geq 1$ for any $(r,\ell) > 1$.
    Suppose $(r,\ell) = 1$.
    By Proposition \ref{prop_decomp_L_quasi_ref}, the action of $g^k$ on $\calV_{n_0}$ is $\pm 1$.
    Furthermore, we may assume $j_0$ as in Proposition \ref{prop_decomp_L_quasi_ref} is $1$.
    Hence, for each case as in Proposition \ref{prop_decomp_L_quasi_ref} we have the lower bound
    \begin{equation}\label{ineq_sigma'}
    \mathrm{RT}'(g^\ell) \geq \sum_{i \geq 2} \bigl\{\ell( a_{i,0}+a_{1,0})\} + \bigl\{2\ell a_{1,0}\}.
    \end{equation}
    By \cite{Gritsenko2007kodaira}*{Proposition 2.7}, $\mathrm{RT}'(g^\ell) \geq 1$ if $\varphi(n_0) > 4$.
    It remains to consider $\varphi(n_0) = 2,3,4$, i.e., $n_0 = 3,4,5,6,8,10,12$.
    A direct calculation using \eqref{ineq_sigma'} and $n \geq 6$ gives $\mathrm{RT}'(g^\ell)\geq1$ for every $\ell$ with $(n_0,\ell)=1$ when $n_0 = 5,8,10,12$.
    For $n_0 = 3,4,6$, we need further computations.
    We will prove the case $n_0 = 6$ and the remaining cases are similar.
    The contribution of $\calV_{n_1}$ to $\mathrm{RT}'(g^\ell)$ is at least $1/3$.
    If $\varphi(n_j) \geq 4$ for some $j>1$, then $\mathrm{RT}'(g^\ell) \geq 1$.
    If $\varphi(n_j) = 1$, then $n_j = 2$ by Proposition \ref{prop_decomp_L_quasi_ref} Case 2 and $2(6, k) = 6$.
    In this case, the contribution of $\calV_{n_j}$ with $j>1$ is at least $1/3$.
        If $\varphi(n_j)=2$, then $n_j=6$ by Proposition \ref{prop_decomp_L_quasi_ref}, Case 2, and $(6,k)=3$.
    In this case, the contribution of $\calV_{n_j}$ to $\mathrm{RT}'(g^\ell)$ is at least $1/3$.
    Hence if $n \geq 5$, then $\mathrm{RT}'(g^\ell) \geq 1$.
    This completes the proof.
\end{proof}

\begin{rem}
    The proof of the above Theorem follows essentially the same argument as in \cite{Gritsenko2007kodaira}.
        The only difference is that one takes into account the one-dimensional subrepresentation $\calV_{n_{j_0}}$ in Proposition \ref{prop_decomp_L_quasi_ref} and the eigenvalue of $g$ on it.
    This additional contribution yields a refinement of \cite{Gritsenko2007kodaira}*{Corollary 2.16}.
\end{rem}

\subsubsection{Boundary}
We now study the boundary of the toroidal compactification $\overline{\calF_L(\Gamma)}$.
We choose the same fan as in \cite{kondo1993kodaira} and \cite{Gritsenko2007kodaira}.
By the construction of toroidal compactifications as in \cite{ash2010smooth}, the boundary components correspond to isotropic subspaces of $L_\bbQ$. 
Since $L_\bbQ$ has the signature $(2,n)$, the dimensions of isotropic subspaces $I$ are at most two.
When $\dim_\bbQ I=1$, the corresponding boundary component $F_I$ lies over a zero-dimensional cusp of the Baily--Borel--Satake compactification of $\calF_L(\Gamma)$.
The following result handles this case.
\begin{thm}[{\cite{Gritsenko2007kodaira}*{Corollary 2.21}, \cite{ma2018kodaira}*{Theorem A.1}}]\label{thm_GHS_0_dim_cusp}
    If $\dim_\bbQ I = 1$, then the boundary $F_I$ has at worst canonical singularities.
\end{thm}
We investigate the singularity for the case  $\dim_\bbQ I = 2$.
The following theorem improves the bound in \cite{Gritsenko2007kodaira}*{Corollary 2.31}, which treats the case $n\geq9$.
\begin{thm}\label{thm_oethogonal_dim_1_cusp}
    If $\dim_\bbQ I = 2$ and $n \geq 6$, then the boundary $F_I$ has at worst canonical singularities.
\end{thm}
Let $F = F_I$ be the boundary component corresponding to $I$ and $P(F)$ be the parabolic subgroup of $\O^+(L_\bbQ)$ stabilizing $I$.
Any element $\gamma$ of $P(F)$ is of the form
\begin{equation}\label{eqn_gamma_P_I}
\gamma=
\begin{pmatrix}
    U_\gamma & V_\gamma & W_\gamma \\
    0 & X_\gamma & Y_\gamma\\
    0&0&Z_\gamma
\end{pmatrix}, \qquad  U_\gamma,Z_\gamma \in \GL_2(\bbQ),\ X_\gamma \in \GL_{n-2}(\bbQ)
\end{equation}
with suitable relations described in \cite{Gritsenko2007kodaira}*{Lemma 2.25}.
Through the realization of $F$ by the Siegel domain of the third kind, we may view $F$ as a quotient of $\bbC^\times \times \bbC^{n-2} \times \frakH$, where $\frakH$ is the upper half space.
After the torus embedding, a boundary point is $(0,\underline{w}, \tau)$.
For $(0,\underline{\omega_0},\tau_0)$ and an element $\gamma$ of its stabilizer, the action of $\gamma$ on the tangent space can be written as
\[
\begin{pmatrix}
    \exp_e(t) & 0 & 0\\
    * & (c\tau_0 + d)^{-1} X_\gamma & 0\\
    * & * & (c\tau_0+d)^{-2}
\end{pmatrix}, \qquad \abcd= Z_\gamma
\]
with some element $\exp_e(t)\in S^1$; for details, see \cite{Gritsenko2007kodaira}*{p.539} and \cite{kondo1993kodaira}*{Subsection 8.2}.
Since $Z_\gamma$ fixes $\tau_0$, $c\tau_0+d$ is a fourth or sixth root of unity, not necessarily primitive.
\begin{lem}\label{lem_action_exp_e}
    We have the following.
    \begin{enumerate}
        \item 
    If $Z_\gamma = \pm 1$ and $X_\gamma^2 = 1$, then $\exp_e(2t) = 1$.
    \item 
    If $(c\tau_0 + d)^{-2} = 1$ and $(c\tau_0 + d)^{-1}X_\gamma =1$, then $\exp_e(t) = 1$.
    \end{enumerate}
\end{lem}
\begin{proof}
    The first statement follows from the latter half of the proof of \cite{Gritsenko2007kodaira}*{Proposition 2.28}.
    The latter statement is merely a definition of $\exp_e(t)$.
    In fact, by $-1 \in \Gamma$, we may assume $c\tau_0 + d =1$ and $X_\gamma = 1$.
    Since $\gamma$ stabilizes some $(0, \underline{\omega_0}, \tau_0)$, we have $\gamma \in U(F)$ in the sense of \cite{Gritsenko2007kodaira}*{Lemma 2.25}.
    Then, the action of $\gamma$ is trivial.
\end{proof}
As a consequence, we may compute the Reid--Shepherd-Barron--Tai sum.
\begin{lem}\label{lem_interior_orthogonal_>1}
Assume that the action of $\gamma$
is nontrivial.
    If $n \geq 6$ and no power of $\gamma$ acts as a quasi-reflection at $(0,\underline{\omega_0},\tau_0)$, then $\mathrm{RT}(\gamma) \geq 1$.
\end{lem}
\begin{proof}
Put $\xi=c\tau_0+d$.
Consider the characteristic polynomial of $X_\gamma$.
If $\xi\neq\pm1$, $\xi$ has order $3$, $4$, or $6$.
A factor of degree $\geq 4$ of the characteristic polynomial contributes at least one
by Lemmas
\ref{lemma_RT_sum_irreducible} and
\ref{lem:external-point-bound}.
Otherwise, every linear factor contributes
at least $1/6$, and every quadratic factor at least $1/3$.
The 
eigenvalue $\xi^{-2}$ contributes at least $1/3$.
Hence $\mathrm{RT}(\gamma)\geq(n-2)/6+1/3=n/6\geq1$.

If $\xi=\pm1$, every nonlinear 
factor contributes at least one by Lemma \ref{lem_RT_sum_A_B_empty}. 
If all factors are linear, $X_\gamma^2=1$ and $\exp_e(t)=\pm1$ by Lemma \ref{lem_action_exp_e},
so all eigenvalues on the tangent space are $\pm1$.
Since the action is nontrivial and is not a
quasi-reflection, at least two eigenvalues are $-1$.
Thus $\mathrm{RT}(\gamma)\geq1$.
\end{proof}
We consider the case where some power of $\gamma$ acts as a quasi-reflection.
Let $k$ be the smallest positive integer such that $\gamma^k$ acts as a quasi-reflection on the tangent space.
Let $\xi = c\tau_0 +d$ and $r_\xi \in \bbQ/\bbZ$ be the exponent of $\xi$.
We decompose $I^\perp/I$ as a $X_\gamma$-module and denote it by
\[
I^\perp/I \cong \bigoplus_j \calV_{n_j}.
\]
We investigate the $\gamma$-module structure of the tangent space.
\begin{lem}\label{lem_xi^2k=1}
In the situation above, the action of
$\xi^{-k}X_\gamma^k$ on $I^\perp/I$
is nontrivial, and
$\xi^{2k}=\exp_e(kt)=1$.
\end{lem}

\begin{proof}
If $\xi^{-k}X_\gamma^k$ is trivial, then
$X_\gamma^k=\xi^k \cdot 1_{n-2}$.
Since $X_\gamma^k$ is a real matrix and
$I^\perp/I$ is nonzero, $\xi^k$ would be
real, hence $\xi^{2k}=1$.
Applying Lemma \ref{lem_action_exp_e} (2) to $\gamma^k$, we have
$\exp_e(kt)=1$.
Hence, $\gamma^k$ acts trivially on the tangent space, 
contrary to the assumption
that $\gamma^k$ is a quasi-reflection.
Thus $\xi^{-k}X_\gamma^k$ is nontrivial.
Since a quasi-reflection has only one
nontrivial eigenvalue, both remaining
one-dimensional blocks are trivial.
This gives $\xi^{2k}=\exp_e(kt)=1$.
\end{proof}

Since $\xi^k=\pm1$, the matrix $\xi^{-k}X_\gamma^k$ is a nontrivial real matrix.
Its unique nontrivial eigenvalue is
therefore real and hence equals $-1$.
Thus $\gamma^{2k}$ acts trivially on the tangent space.
By the minimality of $k$, the action of $\langle \gamma \rangle$ factors through the cyclic group of order $2k$.

Fix $\ell$ with $1 \leq \ell < k$.
We denote the $X_{\gamma^\ell} = X_\gamma^\ell$-module structure of $I^\perp/I$ by $\bigoplus_m \calV_{n_m'}$.

\begin{lem}\label{lem_boundary_orthonal_exp=1}
    Suppose that the action of $\xi^{-k}X_\gamma^k$ on $I^\perp/I$ is nontrivial.
    If $\exp_e(kt) = 1$ and $n \geq 5$, then $\mathrm{RT}'(\gamma^\ell) \geq 1$.
\end{lem}
\begin{proof}
    We claim that if $\varphi(n_j) \geq 2$, then the action of $\xi^{-k}X_\gamma^k$ is trivial on $\calV_{n_j}$.
    Since $\gamma^k$ acts as a quasi-reflection, either $k/n_j - kr_\xi$ or $k(n_j-1)/n_j - k r_\xi$ is an integer.
    Note that $2kr_\xi \in \bbZ$.
    We consider the case that $k/n_j - kr_\xi \in \bbZ$ since the other case can be handled similarly.
Since $k/n_j-kr_\xi\in\bbZ$ and $2kr_\xi\in\bbZ$,
doubling the first relation gives $2k/n_j\in\bbZ$.
Hence $n_j\mid2k$.
    If $n_j$ is odd, then $n_j$ divides $k$ and $kr_\xi \in \bbZ$.
    Hence for any $a \in (\bbZ/n_j\bbZ)^\times$, we have $ak/n_j - kr_\xi \in \bbZ$.
    This implies that the action of $\xi^{-k}X_\gamma^k$ is trivial on $\calV_{n_j}$.
    If $n_j$ is even, then any $a \in (\bbZ/n_j\bbZ)^\times$ can be represented by an odd integer.
    Hence we have
    \[
    \frac{ka}{n_j} - kr_\xi = \frac{k(a-1)}{n_j} + \frac{k}{n_j} - kr_\xi \in \bbZ
    \]
    and the action of $\xi^{-k}X_\gamma^k$ on $\calV_{n_j}$ is trivial.
    This shows the claim.
        It follows that if $\xi^{-k}X_{\gamma}^k$ acts nontrivially on $\calV_{n_j}$, then $\varphi(n_j)=1$.
    We may assume that such $j$ is $1$ and $\calV_{n_1} = \calV_{n_1'}$.

    Suppose that $\xi^{-\ell} = \pm 1$.
        We have already shown that $\mathrm{RT}'(\gamma^\ell)\geq1$ if $\varphi(n_m')\geq2$ for some $m$.
    When $\varphi(n_m') = 1$ for any $m$, the order of $X_\gamma^\ell$ is at most two.
    We have $\exp_e(\ell t) = \pm 1$ by Lemma \ref{lem_action_exp_e} and then all the eigenvalues of $\gamma^\ell$ on the tangent space are $\pm 1$.
    By $\ell < k$, there is no such case.

    Suppose that $\xi^{-\ell} \neq \pm 1$.
        The eigenvalue of $\xi^{-\ell}X_\gamma^\ell$ on $\calV_{n_1'}$ is never $\pm1$, and its contribution to $\mathrm{RT}'$ is at least $1/3$.
    By the same argument as in Lemma \ref{lem_interior_orthogonal_>1}, we have
    \[
    \mathrm{RT}'(\gamma^\ell) \geq \frac13 + \frac13 + \frac{n-3}{6} = \frac{n+1}{6} \geq 1.
    \]
    Here, the first (resp.\ second) $1/3$ is the contribution from $\xi^{-2\ell}$ (resp.\ $\calV_{n_1'}$).
\end{proof}

\begin{proof}[Proof of Theorem \ref{thm_oethogonal_dim_1_cusp}]
 The case in which no power acts as a quasi-reflection
is treated by Lemma \ref{lem_interior_orthogonal_>1}.
Otherwise, Lemmas \ref{lem_xi^2k=1} and
\ref{lem_boundary_orthonal_exp=1} give the required inequalities for $\mathrm{RT}'(\gamma)$
The modified Reid--Shepherd-Barron--Tai criterion \cite{Gritsenko2007kodaira}*{Lemma 2.14} proves the assertion.
\end{proof}

\subsection{Examples of noncanonical singularities}
\label{subsec:An example showing the optimality of the bound}
Let $P$ be the positive-definite even lattice with Gram matrix
\[
  \begin{pmatrix}
    2 & -1\\
    -1 & 2
  \end{pmatrix}
\]
with respect to a basis $e_1,e_2$, and put $L_n
  :=
  P\oplus \langle -2\rangle^{\oplus n}$.
Then $L_n$ has signature $(2,n)$.  Define
\[
  c
  :=
  \begin{pmatrix}
    0 & -1\\
    1 & -1
  \end{pmatrix}
  \in \O^+(P)
\]
and $g_n
  :=
  c\oplus
  \bigl(-\id_{\langle-2\rangle^{\oplus n}}\bigr)
  \in \O^+(L_n)$.
We have $c^3=1$, and hence $g_n$ has order $6$.

Let 
$\omega=e^{2\pi i/3}$ and $w=e_1-\omega e_2$.
A direct calculation gives $c(w)=\omega w$, $(w,w)=0$, and $(w,\overline w)=3>0$.
Thus $[w]\in\mathcal{D}_{L_n}$ and $g_n$ fixes $[w]$.
Put $W=\mathbb Cw$.  Since $W^\perp/W
  \cong
  \langle-2\rangle^{\oplus n}\otimes\mathbb C$,
we have 
  $T_{[w]}\mathcal{D}_{L_n}
  \cong
  \Hom\bigl(W,W^\perp/W\bigr)$.
The element $g_n$ acts on $W$ by $\omega$ and on
$W^\perp/W$ by $-1$.  Therefore it acts on the tangent
space by the scalar $-\omega^{-1}
  =
  e^{2\pi i/6}$.
Consequently, the cyclic quotient singularity at $[w]$ is
of type
\[
  \frac{1}{6}(\underbrace{1,\ldots,1}_{n}).
\]

Choose a torsion-free normal arithmetic
subgroup $\Gamma_0\triangleleft \O^+(L_n)$
and set $\Gamma_n:=\langle\Gamma_0,g_n\rangle$.
Then $\Stab_{\Gamma_n}([w])=\langle g_n\rangle$,
and the image of $[w]$ in
$\Gamma_n\backslash\mathcal{D}_{L_n}$ is a noncanonical singularity if $2\leq n \leq 5$.

\subsection{The case of \texorpdfstring{$\U(1,n)$}{}}

The singularities of ball quotients were studied systematically
by Behrens \cite{behrens2012singularities}.  More recently,
Watson \cite{watson2026singularities} classified the possible noncanonical
cyclic quotient types for ball quotients over the Eisenstein
integers, both in the interior and the boundary.
In particular, his results exhibit noncanonical singularities
in arbitrarily large dimensions over
$\mathbb Q(\sqrt{-3})$.

For this subsection, let $L$ be a
Hermitian $\OO_E$-lattice of signature $(1,n)$.
Let $\mathbb{B}_L
  :=
  \mathcal{D}_{1,n}$
be the associated $n$-dimensional complex ball, and let
$\Gamma\subset \U(L)$ be an arithmetic subgroup.  We write $X_L(\Gamma):=\Gamma\backslash\mathbb B_L$.
When boundary singularities are considered, $X_L(\Gamma)^\Sigma$
denotes the unique toroidal compactification.
The preceding method yields the following refinement of \cite{behrens2012singularities}.

\begin{thm}\label{thm_ball_quot_singularity}
    Assume that $E \neq \bbQ(\sqrt{-1}), \bbQ(\sqrt{-2}), \bbQ(\sqrt{-3})$.
        If the dimension of $X_L(\Gamma)$ satisfies
        \[
        \dim X_L(\Gamma) \geq 
        \begin{cases}
            6 & \text{if $E = \bbQ(\sqrt{-7})$};\\
            5 & \text{if $E \neq \bbQ(\sqrt{-7})$,}
        \end{cases}
        \]
        then the unique toroidal compactification
$\overline{X_L(\Gamma)}$ has at worst canonical singularities.
\end{thm}
\begin{proof}
    The statement follows from Lemmas \ref{lem_ball_interior_non-q-ref}, \ref{lem_ball_interior_q-ref}, \ref{lem_ball_boundary_non-q-ref} and \ref{lem_ball_boundary_q-ref} below combining \cite{behrens2012singularities}*{Lemma 8}.
\end{proof}

In some specific cases of the main theorem in \cite{2025_HMY_Kodaira_dimension_ball_quotients}, the bottleneck is the singularity of ball quotients.

\begin{cor}
    Let $\Gamma\subset\U(L)$.
    Assume that $D>3$ is odd and $n\geq6$ is even with $n+1 \geq B(E)$.
        Then, up to scaling, there exist only finitely many pairs $(L,E)$ for which $X_L(\Gamma)$ is not of general type.
\end{cor}

\begin{cor}
    Suppose that $E \neq \bbQ(\sqrt{-1}), \bbQ(\sqrt{-2}), \bbQ(\sqrt{-3})$ and $\Gamma$ is contained in the discriminant kernel of $\U(L)$.
    If $n + 1 \geq B(E)$ with $n\geq 6$ even, then $X_L(\Gamma)$ is of general type.
\end{cor}

In this subsection below, we assume $E \neq \bbQ(\sqrt{-1}), \bbQ(\sqrt{-2}), \bbQ(\sqrt{-3})$ except for Theorem \ref{thm:noncanonical-sing-ball-quotient}.
We first consider the singularities in the interior.
\begin{lem}\label{lem_ball_interior_non-q-ref}
Assume that the action of $\gamma$
is nontrivial.
    If no power of $\gamma$ acts as a quasi-reflection and 
    \[
    n \geq 
    \begin{cases}
        6 & \text{if $E = \bbQ(\sqrt{-7})$;}\\
        5 & \text{if $E \neq \bbQ(\sqrt{-7})$,}
    \end{cases}
    \]
    then $\mathrm{RT}(\gamma)\geq1$.
\end{lem}
\begin{proof}
Decompose $L_E$ as
\[
L_E = \bigoplus_f W_f
\]
as the $\gamma$-module.
Let $\{a_{f,1},\ldots,a_{f,\deg(f)}\} \subset \bbQ/\bbZ$ be the set of exponents of the roots of $f$.
There exist a polynomial $f_0$ and $i_0$ such that the (multi-)set of exponents of the eigenvalues of the action of $\gamma$ on the tangent space is given by
\[
\left(\bigsqcup_{f}\{ a_{f_0,i_0} - a_{f, j} \mid 1 \leq j \leq \deg(f)\} \right) \setminus \{0\}
\]
and $\mathrm{RT}(\gamma)$ is the sum of $\{a_{f_0,i_0} - a_{f, j}\}$.
For notational simplicity, we may assume $i_0=1$.
By Table \ref{table_list_c<=1}, if $\deg(f_0) \geq 5$, then $\mathrm{RT}(\gamma) \geq c_{\min}(f_0, E) \geq 1$.
We treat each possible value of $\deg(f_0)$ separately.

For the case $\deg(f_0) = 4$, it remains to consider the cases 
\[
(f_0, E) = (\Phi_{15}^\pm, \bbQ(\sqrt{-15})), (\Phi_{30}^\pm, \bbQ(\sqrt{-15})), (\Phi_{20}^\pm, \bbQ(\sqrt{-20})), (\Phi_{24}^\pm, \bbQ(\sqrt{-24})).
\]
We treat only $(f_0,E)=(\Phi_{15}^\pm,\bbQ(\sqrt{-15}))$; the remaining cases are analogous.
The contribution of $f_0$ to $\mathrm{RT}(\gamma)$ is at least $11/15 > 2/3$.
Hence, by Lemma \ref{lemma_S_XA_B}, if $\deg(f) \geq 2$ for some $f \neq f_0$, then $\mathrm{RT}(\gamma) > 1$ and it concludes that we may assume $a_{f,1} = 0,1/2$ for $f \neq f_0$.
The case $(\Phi_{15}^+,\bbQ(\sqrt{-15}))$ satisfies $\mathrm{RT}(\gamma)>1$, since the exponents of the roots of $\Phi_{15}^+$ are $\{1/15,2/15,4/15,8/15\}$.
For the case $(\Phi^-_{15}, \bbQ(\sqrt{-15}))$, we have $\mathrm{RT}(\gamma) > 1$ except for $a_{f_0,1} = 14/15$.
In this case, since other $a_{f,j}$ is $0$ or $1/2$, the contribution of $f \neq f_0$ is at least $13/30$ and we have
\[
\mathrm{RT}(\gamma) \geq \frac{11}{15} + \frac{13}{30}(n-3) > 1
\]
due to the assumption $n \geq 5$.

For the case $\deg(f_0) = 3$, it remains to consider the cases
\[
(f_0,E) = (\Phi_7^{\pm}, \bbQ(\sqrt{-7})), (\Phi_{14}^{\pm}, \bbQ(\sqrt{-7})).
\]
Similar to the above case, if $\deg(f) \geq 2$ for some $f \neq f_0$, then $\mathrm{RT}(\gamma) > 1$.
When $f_0=\Phi_7^\pm$, the minimum occurs when $a_{f_0,1}=4/7$, and then
\[
\mathrm{RT}(\gamma) \geq \left\{\frac47 - \frac17\right\} + \left\{\frac47-\frac27\right\} + (n-2)\left\{\frac47 -\frac12\right\} = \frac{n+8}{14} \geq 1
\]
due to $n \geq 6$.
Here the exponent set of $\Phi_7^+$ is $\{1/7,2/7,4/7\}$.
The other case is similar.

Suppose that $\deg(f_0)=2$.
If another factor $f$ has degree at least
four, Lemma \ref{lem:external-point-bound}
and Table \ref{table_list_c<=1} give a
contribution at least
$c_{\min}(f,E)\geq11/15$ from that factor.
Together with the contribution of at
least $1/3$ from $f_0$, $\mathrm{RT}(\gamma)$ is greater
than one.
In this remaining case, $f_0$ is one of
$\Phi_3,\Phi_4,\Phi_6$.
Its contribution is at least $1/3$.
Every remaining linear factor has exponent $0$ or $1/2$
and contributes at least $1/6$.
For each remaining quadratic factor
$\Phi_3,\Phi_4,\Phi_6$, direct summation of its two
roots gives a contribution at least $1/3$.
The total dimension outside $W_{f_0}$ is $n-1$.
Consequently,
$\mathrm{RT}(\gamma)\geq1/3+(n-1)/6=(n+1)/6\geq1$.
This completes the case $\deg(f_0)=2$.

For a factor of degree three, the only
remaining case is $E=\bbQ(\sqrt{-7})$,
with $f=\Phi_7^\pm$ or $\Phi_{14}^\pm$.
Directly summing the exponents for
$f_0=\Phi_3,\Phi_4,\Phi_6$ gives a
combined contribution of at least $4/3$.
Hence we may assume $\deg(f)\leq2$
for every $f\neq f_0$.

For the case $\deg(f_0) = 1$, we have $\mathrm{RT}(\gamma) \geq 1$ if $\deg(f) \geq 2$ for some $f \neq f_0$.
When every $f$ satisfies $\deg(f)=1$, any nontrivial eigenvalue contributes $1/2$ to $\mathrm{RT}(\gamma)$.
Since $\gamma$ is not a quasi-reflection, there are at least two nontrivial eigenvalues.
Hence we have $\mathrm{RT}(\gamma) \geq 1$.
\end{proof}

Suppose that there exists a positive integer $k$ such that $\gamma^k$ acts as a quasi-reflection.
We may choose $k$ to be minimal.
Let $m_f$ be the positive integer such that $f = \Phi_{m_f}$ or $\Phi_{m_f}^\pm$.
The following is essentially due to \cite{behrens2012singularities}*{Proposition 2} and its proof \cite{behrens2012singularities}*{p.399 (3)}.

\begin{prop}[{\cite{behrens2012singularities}*{Proposition 2}}]\label{prop_ball_q-ref}
    As a $\gamma$-module, if $n\geq2$, the decomposition $L_E=\bigoplus_f W_f$ satisfies one of the following two conditions:
    \begin{itemize}
        \item There exists $f_1 \neq f_0$ such that $2(m_{f_1}, k) = m_{f_1}$ and $(m_f, k) = m_f$ for any $f \neq f_1$;
        \item There exists $f_1 \neq f_0$ such that $(m_{f_1}, k) = m_{f_1}$ and $2(m_f, k) = m_f$ for any $f \neq f_1$.
    \end{itemize}
    Furthermore, $m_{f_1} = 2$ (resp. $m_{f_1} = 1$) for the first case (resp. the second case).
\end{prop}
\begin{proof}   
    Every statement other than the last follows from \cite{behrens2012singularities}*{Proposition 2}.
    The action of $\gamma^k$ on $W_{f_0}$ has the eigenvalues $\pm 1$ by $n \geq 2$ and the proof of \cite{behrens2012singularities}*{Proposition 2}.
    Suppose that $m_{f_0}$ is odd.
        In this case, since no power of a primitive $m_{f_0}$-th root of unity is $-1$, $\gamma^k$ acts trivially on $W_{f_0}$.
    If $m_{f_0}$ is even, any exponent of primitive $m_{f_0}$-th root of unity is written as $a/m_{f_0}$ with an odd integer $a$ and we have $\{ka/m_{f_0}\} = 0,1/2$.
    This shows that $\{ka/m_{f_0}\}$ is independent of $a$ and that the action $\gamma^k$ on $W_{f_0}$ is scalar $\pm 1$.
    Hence, the nontrivial eigenvalue of $\gamma^k$ on the tangent space comes from $W_{f_1}$ with $f_1 \neq f_0$.
        Applying \cite{behrens2012singularities}*{p.399 (3)} to $W_{f_1}$ and using the fact that $\gamma^{2k}$ acts as trivial by \cite{behrens2012singularities}*{Corollary 3}, it follows that $\dim_EW_{f_1}=1$.
\end{proof}

Consequently, $\gamma$ has order $2k$ if $E\neq\bbQ(\sqrt{-1}),\bbQ(\sqrt{-2}),\bbQ(\sqrt{-3})$.

\begin{lem}\label{lem_ball_interior_q-ref}
    If $n\geq5$ and $\gamma^k$ acts as a quasi-reflection, then $\mathrm{RT}'(\gamma^\ell)\geq1$ for every $1\leq\ell<k$.
\end{lem}
\begin{proof}
    The proof is similar to that of Lemma \ref{lem_ball_interior_non-q-ref}.
    We decompose $L_E$ as a $\gamma^\ell$-module and denote it by
    $\bigoplus_{f'} W_{f'}$.
        We define $f'_0$ and $f'_1$ analogously.
    If $\deg(f_0') \geq 4$, then the contribution of $W_{f_0'} \oplus W_{f_1'}$ to $\mathrm{RT}'(\gamma^\ell)$ is greater than $1$.
    If $\deg(f_0') = 3$, it suffices to consider the cases $(f_0',E) = (\Phi_7^\pm, \bbQ(\sqrt{-7})), (\Phi_{14}^\pm, \bbQ(\sqrt{-7}))$.
        In the worst case, we obtain
    \[
    \mathrm{RT}'(\gamma^\ell) \geq \frac{5}{7} + 2 \cdot \frac{1}{14} + \frac{n-3}{14} = \frac{n+9}{14} \geq 1.
    \]
    If $\deg(f_0') = 2$, then the contribution of $W_{f_0'} \oplus W_{f_1'}$ to $\mathrm{RT}'(\gamma^\ell)$ is at least $2/3$.
        By the same argument as in the proof of Lemma \ref{lem_ball_interior_non-q-ref}, we have
    \[
    \mathrm{RT}'(\gamma^\ell) \geq \frac{2}{3} + \frac{n-3}{6} = \frac{n+1}{6} \geq 1.
    \]
    If $\deg(f_0') = 1$ and $\deg(f') \geq 2$ for some $f' \neq f'_0$, then $\mathrm{RT}'(\gamma^{\ell}) \geq s_{\min}(f', E) \geq 1$ by Table \ref{table_list_S(X,0)<=1} and the assumption that $E \neq \bbQ(\sqrt{-1}), \bbQ(\sqrt{-2}), \bbQ(\sqrt{-3})$.
    If all $f'$ are linear, then all the eigenvalues are $\pm 1$.
By $\ell < k$, this case does not happen.
\end{proof}

The boundary case is simpler. Let $I\subset L_E$ be an isotropic line.
A neighborhood of the corresponding boundary component is modeled on
\[
\bbC^\times\times((I^\perp/I)\otimes_E\bbC)
\cong \bbC^\times\times\bbC^{n-1},
\]
and the boundary divisor corresponds to $\{0\}\times\bbC^{n-1}$.
Fix $z_0=(0,\underline{\omega_0})$ and let $\gamma$ lie in its stabilizer.
Following the notation \eqref{equation_action_d_gamma}, the action of $\gamma$ on the tangent space is given by
\[
d_{\gamma_{z_0}} = \begin{pmatrix}
    D_{\gamma,z_0}^{\mathrm{tor}} & * \\
    0 & (v \mapsto a_\gamma v d_\gamma^*)
\end{pmatrix}
\]
and $a_\gamma \in \{\pm 1\}$.
We decompose a $d_\gamma$-module $I^\perp/I$ by $\bigoplus_f W_f$.
Let $m_f$ be the unique positive integer such that $f = \Phi_{m_f}$ or $\Phi_{m_f}^\pm$.
\begin{lem}\label{lem_action_translation}
If $d_\gamma$ has order at most two,
then the action of $\gamma$ on the
tangent space has eigenvalues $\pm1$.
\end{lem}
\begin{proof}
    This is proved in \cite{behrens2012singularities}*{Proposition 4}.
    See also the proof of \cite{Gritsenko2007kodaira}*{Proposition 2.28}.
\end{proof}

\begin{lem}\label{lem_ball_boundary_non-q-ref}
If 
$\gamma$ acts on the tangent space nontrivially and no power of it
acts as a quasi-reflection, then
$\mathrm{RT}(\gamma)\geq1$.
\end{lem}
\begin{proof}
    By Table \ref{table_list_S(X,0)<=1} and the assumption $E \neq \bbQ(\sqrt{-1}),  \bbQ(\sqrt{-2}), \bbQ(\sqrt{-3})$, if $\deg(f) \geq 2$ for some $f$, then $\mathrm{RT}(\gamma) \geq s_{\min}(f,E) \geq 1$.
    We may assume $\deg(f) \leq 1$ for any $f$.
    The eigenvalues of the action of $\gamma$ on the tangent space is $\pm 1$ by Lemma \ref{lem_action_translation}.
    Since $\gamma$ does not act as a quasi-reflection, $\mathrm{RT}(\gamma) \geq 1$ is immediate.
\end{proof}

\begin{lem}\label{lem_ball_boundary_q-ref}
Assume $n\geq3$, and let $k$ be the least positive
integer such that $\gamma^k$ acts as a quasi-reflection.
Then the  action of $\gamma$ has
order $2k$, and
$\mathrm{RT}'(\gamma^\ell)\geq1$
for every $1\leq\ell<k$.
\end{lem}

\begin{proof}
By Lemma \ref{lem:boundary-interior-comparison}, $D_{\gamma^k,z_0}^\tor = 1$ and the action $v \mapsto a_\gamma^k v (d_\gamma^*)^k$ is nontrivial with exactly one nontrivial eigenvalue. 
The eigenspace $W$ with the nontrivial eigenvalue is one-dimensional and defined over $E$ with eigenvalue $\OO_E^\times = \{\pm 1\}$.
Hence, $\gamma^{2k}$ acts as trivial on the tangent space.

If there is a nonlinear factor dividing the characteristic polynomial of $d_\gamma^\ell$, we have $\mathrm{RT}'(\gamma^{\ell}) \geq 1$ since $\dim W = 1$.
If all factors are linear, then
$(d_\gamma^\ell)^2=1$.
By Lemma \ref{lem_action_translation}, all eigenvalues of $\gamma^\ell$ are $\pm1$.
This is impossible for $1\leq\ell<k$.
This proves the assertion.
\end{proof}

Finally, note that the assumption $E\neq\bbQ(\sqrt{-1}),\bbQ(\sqrt{-2}),\bbQ(\sqrt{-3})$ is natural.
In fact, we obtain the following.
\begin{thm}\label{thm:noncanonical-sing-ball-quotient}
Let $E=\bbQ(\sqrt{-1}),\bbQ(\sqrt{-2})$, or
$\bbQ(\sqrt{-3})$.
For every $n\geq2$, there exist a Hermitian $\OO_E$-lattice
$L_n$ of signature $(1,n)$ and a subgroup
$\Gamma_n\subset\U(L_n)$ of finite index such that
$X_{L_n}(\Gamma_n)$ has a noncanonical interior singularity.
\end{thm}

\begin{proof}
We use the realization of
$\mathbb B_{L_n}=\mathcal D_{1,n}$ by negative $n$-planes
in \eqref{def_realization_D_p,q}.
In each case below, put $W=\bbC e_0$ and let
$x_n=W^\perp\in\mathbb B_{L_n}$.
Thus
$T_{x_n}\mathbb B_{L_n}\cong
\operatorname{Hom}_{\bbC}(W^\perp,W)$.
The elements constructed below act trivially on $W$,
so their eigenvalues on the tangent space are the complex conjugate of their eigenvalues on $W^\perp$.
Write $\zeta_m=\exp(2\pi\sqrt{-1}/m)$.

\medskip
\noindent\textbf{Case $E=\bbQ(\sqrt{-1})$.}
Take $L_n=\OO_E^{n+1}$ with Hermitian form
$|z_0|^2-|z_1|^2-\cdots-|z_n|^2$, and set
\[
  \gamma_n
  =\operatorname{diag}
  (1,-\sqrt{-1},-\sqrt{-1},1_{n-2})
  \in\U(L_n).
\]
Its eigenvalues at $x_n$ are
$\sqrt{-1},\sqrt{-1},1,\ldots,1$.
The cyclic quotient type is
$\frac14(1,1,0,\ldots,0)$, and
$\mathrm{RT}(\gamma_n)=1/2<1$.

\medskip
\noindent\textbf{Case $E=\bbQ(\sqrt{-2})$.}
Put
\[
  H=
  \begin{pmatrix}
    2&-\sqrt{-2}\\
    \sqrt{-2}&2
  \end{pmatrix},
  \qquad
  A=
  \begin{pmatrix}
    \sqrt{-2}&1\\
    1&0
  \end{pmatrix}.
\]
The matrix $H$ is positive definite, and
$A^*HA=H$, $A^4=-1_2$, and $A^8=1_2$.
The eigenvalues of $A$ are $\zeta_8,\zeta_8^3$, since
$\det(T\cdot1_2-A)=T^2-\sqrt{-2}\,T-1$.
Equip
$L_n=\OO_E e_0\oplus\OO_E^2\oplus\OO_E^{n-2}$
with the Hermitian form of Gram matrix
$\operatorname{diag}(1,-H,-1_{n-2})$, and set
\[
  \gamma_n=1\oplus A^{-1}\oplus1_{n-2}\in\U(L_n),
  \qquad
  A^{-1}=
  \begin{pmatrix}
    0&1\\
    1&-\sqrt{-2}
  \end{pmatrix}.
\]
The eigenvalues at $x_n$ are
$\zeta_8,\zeta_8^3,1,\ldots,1$.
Thus the type is $\frac18(1,3,0,\ldots,0)$, and
$\mathrm{RT}(\gamma_n)=1/2<1$.

\medskip
\noindent\textbf{Case $E=\bbQ(\sqrt{-3})$.}
Take $L_n=\OO_E^{n+1}$ with the standard Hermitian form
of signature $(1,n)$, and set
\[
  \gamma_n
  =\operatorname{diag}
  (1,\zeta_3^{-1},\zeta_3^{-1},1_{n-2})
  \in\U(L_n).
\]
Its eigenvalues at $x_n$ are
$\zeta_3,\zeta_3,1,\ldots,1$.
The type is $\frac13(1,1,0,\ldots,0)$, and
$\mathrm{RT}(\gamma_n)=2/3<1$.

\medskip
In each case, choose a torsion-free normal subgroup
$\Gamma_0$ of finite index in $\U(L_n)$, and put
$\Gamma_n=\langle\Gamma_0,\gamma_n\rangle$.
We claim that
$\operatorname{Stab}_{\Gamma_n}(x_n)=\langle\gamma_n\rangle$.
Indeed, normality implies that every element of $\Gamma_n$
has the form $\eta\gamma_n^j$ with $\eta\in\Gamma_0$.
If this element fixes $x_n$, then $\eta$ also fixes $x_n$.
Since the stabilizer of an interior point is finite
and $\Gamma_0$ is torsion-free, we have $\eta=1$.

The two nontrivial eigenvalues on the tangent space are both primitive
roots of the order of $\gamma_n$, namely $4$, $8$, or $3$.
Thus every nontrivial element of $\langle\gamma_n\rangle$
has exactly two nontrivial eigenvalues on the tangent space,
and no such element is a quasi-reflection.
The Reid--Shepherd-Barron--Tai criterion Theorem \ref{thm:Reid--Shepherd-Barron--Tai} therefore applies,
and the above strict inequalities show that the image of
$x_n$ in $X_{L_n}(\Gamma_n)$ is noncanonical.
\end{proof}
Note that the example introduced in the case $E=\Q(\sqrt{-3})$ above is also written in  \cite{watson2026singularities}*{Example 4.3.6}.

\newpage
\appendix

\section{Computer-assisted finite checks}
\label{app:finite-computations}
\label{subsec:finite-computations}
\setcounter{algorithm}{0}
\renewcommand{\thealgorithm}{\thesection.\arabic{algorithm}}
\providecommand{\theHalgorithm}{\thealgorithm}
\renewcommand{\theHalgorithm}{appendix.\thesection.\arabic{algorithm}}
\setcounter{table}{0}
\renewcommand{\thetable}{\thesection.\arabic{table}}
\providecommand{\theHtable}{\thetable}
\renewcommand{\theHtable}{appendix.\thesection.\arabic{table}}

This appendix records the finite enumerations used in Sections \ref{section_auxiliary_lemma}--\ref{sec:boundary-singularities}. All inputs and intermediate quantities are rational numbers, and no numerical approximation is used.

For a positive integer $d$, put
\[
X(\Phi_d)\defeq\left\{\frac{a}{d}\bmod1\ \middle|\ 1\leq a\leq d,\ (a,d)=1\right\}.
\]
We use the notation of Subsection \ref{subsec:cyclic-reps}. If the fundamental discriminant of $E$ is $-D$ and $D\mid d$, let $\omega'_{E,d}:(\bbZ/d\bbZ)^\times\to\{\pm1\}$
be the pullback of the quadratic character associated with $E$. For $\epsilon\in\{\pm1\}$, define
\[
X_{E,\epsilon}(d)\defeq\left\{\frac{a}{d}\bmod1\ \middle|\ (a,d)=1,\ \omega'_{E,d}(a)=\epsilon\right\}.
\]
These are the sets of exponents of the two irreducible factors $\Phi_d^\pm$ over $E$. If $D\nmid d$, the polynomial $\Phi_d$ remains irreducible over $E$ and its corresponding set is $X(\Phi_d)$.
Finally, let
\[
U_E\defeq
\begin{dcases}
\left\{0,\frac{1}{2}\right\}, & \text{if $w_E=2$};\\
\left\{0,\frac{1}{4},\frac{2}{4},\frac{3}{4}\right\}, & \text{if $w_E=4$};\\
\left\{0,\frac{1}{6},\frac{2}{6},\frac{3}{6},\frac{4}{6},\frac{5}{6}\right\}, & \text{if $w_E=6$}.
\end{dcases}
\]

After the elementary estimates in Lemmas
\ref{lemma_RT_sum_irreducible} and \ref{lemma_RT_sum_reducible},
only finitely many $d$ remain to be checked directly.
For the irreducible case of Lemma \ref{lemma_RT_sum_irreducible}, after
the cases $d=1,2,3,4,6$ treated explicitly in its proof, the remaining values are $\mathcal{D}_{\mathrm{irr}}
:=
\{8,10,12,18,30\}$.
For the reducible case of Lemma \ref{lemma_RT_sum_reducible}, the
remaining $d$ are
\[
\begin{split}
\mathcal{D}_{\mathrm{red}}
:=\{&7,8,9,11,12,14,15,16,18,20,21,22,24,28,30,32,36,38,40,42,\\
&48,54,60,66,70,72,78,84,90\}.
\end{split}
\]
For a uniform finite enumeration below, we put
\[
\begin{split}
\mathcal{D}_{\mathrm{exc}}
:=\mathcal{D}_{\mathrm{irr}}\cup\mathcal{D}_{\mathrm{red}}
=\{&7,8,9,10,11,12,14,15,16,18,20,21,22,24,28,30,32,36,38,40,42,\\
&48,54,60,66,70,72,78,84,90\}.
\end{split}
\]

\noindent \textbf{Minimum along bipartitions.}
The first algorithm is the elementary subroutine used throughout the appendix. Its input is a finite multiset in $\bbQ/\bbZ$, and its output is the  minimum of the Reid--Shepherd-Barron--Tai sum over all nontrivial decompositions of that multiset. Since the set of decompositions is finite, the enumeration is exhaustive; the decompositions providing the minimum are recorded in order to identify the equality cases used in the main text.

\begin{algorithm}[H]
\caption{Minimum over non-trivial bipartitions}\label{alg:min-partition}
\begin{algorithmic}[1]
\Require A finite multiset $X=\{x_1,\ldots,x_n\}\subset \mathbb Q/\mathbb Z$ with $n\geq2$.
\Ensure The value $c(X)=\min_{X=A\sqcup B}S(A,B)$ and the set of minimizing decompositions.
\State $c\gets +\infty$, $\mathcal W\gets\emptyset$
\For{each nonempty proper subset of indices $I\subset \{1,\ldots,n\}$}
  \State $A\gets\{x_i:i\in I\}$, $B\gets X\setminus A$
  \State $v\gets \sum_{a\in A,b\in B}\{a-b\}$
  \If{$v<c$}
    \State $c\gets v$, $\mathcal W\gets\{(A,B)\}$
  \ElsIf{$v=c$}
    \State Add $(A,B)$ to $\mathcal W$
  \EndIf
\EndFor
\State \Return $(c,\mathcal W)$
\end{algorithmic}
\end{algorithm}

\noindent \textbf{Exceptional values of $c_{\min}$.}
After the elementary estimates of Section \ref{section_auxiliary_lemma}, only the finite set $\mathcal{D}_{\mathrm{exc}}$ can contribute exceptional values of $c_{\min}$.  The following enumeration takes each such $d$, forms either $X(\Phi_d)$ or $X_{E,\epsilon}(d)$,
and applies Algorithm \ref{alg:min-partition}.  
\begin{algorithm}[H]
\caption{Exceptional values of $c_{\min}$}\label{alg:cmin-exceptional}
\begin{algorithmic}[1]
\Require The list $\mathcal{D}_{\mathrm{exc}}$.
\Ensure The values $c_{\min}$ for all  $X(\Phi_d)$ and $X_{E,\epsilon}(d)$
occurring in Lemmas \ref{lemma_RT_sum_irreducible} and
\ref{lemma_RT_sum_reducible}.
\State $\mathcal R\gets\emptyset$
\For{each $d\in\mathcal{D}_{\mathrm{exc}}$}
  \State Apply Algorithm \ref{alg:min-partition} to $X(\Phi_d)$ and append the output to $\mathcal R$
  \For{each imaginary quadratic fundamental discriminant $-D$ with $D\mid d$}
    \For{each $\epsilon\in\{\pm1\}$}
      \State $X\gets X_{E,\epsilon}(d)$
      \If{$|X|\geq2$}
        \State Apply Algorithm \ref{alg:min-partition} to $X$
        \State Append $(E,d,\epsilon,c,\mathcal W)$ to $\mathcal R$
      \EndIf
    \EndFor
  \EndFor
\EndFor
\State \Return $\mathcal R$
\end{algorithmic}
\end{algorithm}

\noindent \textbf{Exceptional values of $s_{\min}$.}
The next enumeration treats the case in which one of the two eigenspaces is contributed by a unit of $\OO_E^\times$.  The input is the finite list of $\Phi_d$ and $\Phi_d^\pm$ not already removed by the preceding estimates, together with $U_E$.  For each factor, the algorithm compares the two possible contributions $S(X_f,\{\theta\})$ and $S(\{\theta\},X_f)$ for all $\theta\in U_E$, and records the minimum.
This produces the values in Table \ref{table_list_S(X,0)<=1}.

By the proof of Lemma \ref{lem_RT_sum_A_B_empty}, it is enough to
consider $d$ in
\[
\mathcal D_s
:=
\{3,4,6,7,8,9,12,14,15,18,20,24,30\}.
\]
The symbol ``$\mathrm{gen}$'' in Algorithm
\ref{alg:smin-exceptional} denotes any imaginary quadratic field
$E\neq\bbQ(\sqrt{-1}),\bbQ(\sqrt{-3})$ for which $D\nmid d$.
Second, we enumerate separately the finitely many exceptional fields for which either the unit group is larger or $\Phi_d$ is reducible.

\begin{algorithm}[H]
\caption{Exceptional values of $s_{\min}(f,E)$}
\label{alg:smin-exceptional}
\begin{algorithmic}[1]
\Require The finite set $\mathcal D_s$.
\Ensure The values in Table
\ref{table_list_S(X,0)<=1} and Lemma
\ref{lem_RT_sum_A_B_empty}, including the generic irreducible case.

\State $\mathcal R\gets\emptyset$

\For{each $d\in\mathcal D_s$}

  \Statex \textit{Generic irreducible cases.}
  \State
  $X_{\mathrm{gen}}\gets X(\Phi_d)$,
  $U_{\mathrm{gen}}\gets\{0,1/2\}$

  \If{$|X_{\mathrm{gen}}|\geq2$}
    \State
    $s\gets+\infty$, $\mathcal W\gets\emptyset$

    \For{each $\theta\in U_{\mathrm{gen}}$}
      \State
      $v_1\gets S(X_{\mathrm{gen}},\{\theta\})$,
      $v_2\gets S(\{\theta\},X_{\mathrm{gen}})$
      \State Update $s$ and $\mathcal W$ using $v_1$ and $v_2$
    \EndFor

    \State Append
    $(\mathrm{gen},d,X_{\mathrm{gen}},s,\mathcal W)$
    to $\mathcal R$
  \EndIf

  \Statex \textit{Exceptional cases.}
  \State Let $\mathcal E_{\mathrm{exc}}(d)$ consist of
  $\bbQ(\sqrt{-1})$, $\bbQ(\sqrt{-3})$, and all imaginary quadratic
  fields of fundamental discriminant $-D$ satisfying $D\mid d$

  \For{each $E\in\mathcal E_{\mathrm{exc}}(d)$}

    \If{$D\mid d$}
      \State
      $\mathcal X\gets\{X_{E,+}(d),X_{E,-}(d)\}$
    \Else
      \State
      $\mathcal X\gets\{X(\Phi_d)\}$
    \EndIf

    \For{each $X\in\mathcal X$}

      \If{$|X|\geq2$}
        \State
        $s\gets+\infty$, $\mathcal W\gets\emptyset$

        \For{each $\theta\in U_E$}
          \State
          $v_1\gets S(X,\{\theta\})$,
          $v_2\gets S(\{\theta\},X)$
          \State Update $s$ and $\mathcal W$ using $v_1$ and $v_2$
        \EndFor

        \State Append $(E,d,X,s,\mathcal W)$ to $\mathcal R$
      \EndIf

    \EndFor
  \EndFor
\EndFor

\State \Return $\mathcal R$
\end{algorithmic}
\end{algorithm}

\noindent\textbf{Pairwise exceptional inequalities.}
By Lemma \ref{lem:external-point-bound}, we have
$s_{\min}(f,E)\geq c_{\min}(f,E)$.
Hence, to prove Lemmas \ref{lemma_S_AA_BB},
and \ref{lemma_S_XX}, it is enough to consider 
$f$ 
satisfying
$\deg(f)\geq2$ and $c_{\min}(f,E)\leq1$.
By Lemmas \ref{lemma_RT_sum_irreducible} and
\ref{lemma_RT_sum_reducible}, $f$ is equal to $\Phi_d$ or $\Phi_{d}^\pm$ for some $d \in \calD_s$.

Let $\mathcal E_{\mathrm{pair}}$ consist of one irreducible cases together with all imaginary quadratic fields
of fundamental discriminant $-D$ such that $D\mid d$ for some
$d\in\mathcal D_s$.
For $E\in\mathcal E_{\mathrm{pair}}$, let
$\mathcal F_{\mathrm{pair}}(E)$ be the finite collection of
irreducible factors $f$ of $\Phi_d$, $d\in\mathcal D_s$,
satisfying $\deg(f)\geq2$ and $c_{\min}(f,E)\leq1$.
If $D\nmid d$, we take $f=\Phi_d$ and $X_f=X(\Phi_d)$; if $D\mid d$,
we use the factors $\Phi_d^\pm$ with
$X_{\Phi_d^\pm}=X_{E,\pm}(d)$.
For the generic field, all $\Phi_d$ are treated as irreducible and
$c_{\min}(\Phi_d,\mathrm{gen})=c(X(\Phi_d))$.

For a finite multiset $X$, put
$\operatorname{Part}(X)
:=\{(A,B)\mid X=A\sqcup B\text{ as multisets},\
A\neq\emptyset,\ B\neq\emptyset\}$.
These are exactly the quantities required in
Lemmas \ref{lemma_S_AA_BB}, and
\ref{lemma_S_XX}, respectively. 
\begin{algorithm}[H]
\caption{Pairwise finite checks in Section
\ref{sec:ages_unitary_actions}}
\label{alg:exceptional-pair}
\begin{algorithmic}[1]
\Require The finite collections $\mathcal F_{\mathrm{pair}}(E)$,
$E\in\mathcal E_{\mathrm{pair}}$.
\Ensure The assertions of Lemmas \ref{lemma_S_AA_BB}
and \ref{lemma_S_XX}.

\State $\mathcal R\gets\emptyset$
\For{each $E\in\mathcal E_{\mathrm{pair}}$ and each ordered pair
$f,f'\in\mathcal F_{\mathrm{pair}}(E)$}
  \For{each $(A,B)\in\operatorname{Part}(X_f)$}
    \For{each $(A',B')\in\operatorname{Part}(X_{f'})$}
      \State Compute $S(A\sqcup A',B\sqcup B')$
      and record it if it is at most $1$
    \EndFor
  \EndFor
  \State Compute $S(X_f,X_{f'})$
  and record it if it is at most $1$
\EndFor
\State \Return $\mathcal R$
\end{algorithmic}
\end{algorithm}

\begin{rem}
\label{rem:unit-interior-enumeration}
Let $A,B$ be multisets in $U_E$ with
$2\leq |A|=p\leq |B|=q$ and $0<S(A,B)\leq1$.
By Lemma \ref{lem:unit-eigenvalues-interior}, $p\leq w_E$.
If $A$ is non-scalar, every $b\in B$ contributes at least
$1/w_E$, so $q\leq w_E$.
If $A$ consists of $p$ copies of $\alpha$, remove all occurrences
of $\alpha$ from $B$ and denote the resulting multiset by
$B_{\mathrm{ess}}$.
Then $S(A,B)
  =p\sum_{b\in B_{\mathrm{ess}}}\{\alpha-b\}
  \geq p|B_{\mathrm{ess}}|/w_E$ and $
  1\leq |B_{\mathrm{ess}}|
  \leq\left\lfloor\frac{w_E}{p}\right\rfloor$.
These bounds give a finite enumeration.
The full families are recovered by adjoining copies of $\alpha$
to $B_{\mathrm{ess}}$ until $|B|\geq p$; further copies add only zeros for types.
\end{rem}

\noindent\textbf{Types with nonlinear factors in the interior.}
The cases with at least two nonlinear factors are supplied by
Proposition \ref{prop_two_nonlinear_factor}.
Also, by Proposition \ref{prop_one_nonlinear_factor}, if $A_f = \emptyset$ or $B_f = \emptyset$, we have $\mathrm{RT}(\gamma) \geq 1$ and the equality holds only if $s_{\min}(f,E) = 1/2$ and $p=2$.
We may assume $A_f \neq \emptyset$ and $B_f \neq \emptyset$.
Put
$\mathcal F_{\mathrm{one}}
:=\{(E,f)\mid
E\in\mathcal E_{\mathrm{pair}},
f\in\mathcal F_{\mathrm{pair}}(E)\}$.
For $(E,f)\in \mathcal F_{\mathrm{one}}$, set
\[
  M_{E,f} 
  \defeq
  \max_{x,y \in X_f \cup U_E, x \neq y}
  \left\lfloor\frac{1}{\{x-y\}}\right\rfloor.
\]
If $S(A,B) \leq 1$, we have $|A|, |B| \leq M_{E,f}$.
\begin{algorithm}[H]
\caption{Interior low-age types with nonlinear factors}
\label{alg:low-age-types}
\begin{algorithmic}[1]
\Require The finite set $\mathcal F_{\mathrm{one}}$.
\Ensure All interior cyclic quotient types with a nonlinear factor
and age at most $1$.

\State $\mathcal R\gets\emptyset$

\For{each $(E,f)$ with $E \in \calE_{\mathrm{pair}}$ and $f \in \calF_{\mathrm{pair}}(E)$}
  \For{each decomposition $X_f=A_f\sqcup B_f$, not allowing empty parts}
    \For{unit multisets $U^+,U^-\subset U_E$ with
      $|U^+|,|U^-|\leq M_{E,f}$}
      \State $A\gets A_f\sqcup U^+$,\quad $B\gets B_f\sqcup U^-$
      \If{$2\leq|A|\leq|B|$}
        \State $R\gets S(A,B)$
        \If{$0<R\leq1$}
          \State $T\gets\{\{a-b\}\mid a\in A,\ b\in B\}$
          \State Append $(E,f,|A|,|B|,R,T)$ to $\mathcal R$
        \EndIf
      \EndIf
    \EndFor
  \EndFor
\EndFor

\State Add the degree-two equality cases of Proposition
\ref{prop_two_nonlinear_factor} by direct enumeration in signature
$(2,2)$.

\State \Return $\mathcal R$
\end{algorithmic}
\end{algorithm}

\small
\begin{longtable}{
  >{\raggedright\arraybackslash}p{0.22\textwidth}
  >{\raggedright\arraybackslash}p{0.45\textwidth}
  >{\raggedright\arraybackslash}p{0.23\textwidth}
}
\caption{Where the finite enumerations are used in the main text}
\label{tab:computer-assisted-map}\\
\toprule
Statement & Required output & Algorithm \\
\midrule
\endfirsthead

\toprule
Statement & Required output & Algorithm \\
\midrule
\endhead

Lemma \ref{lemma_RT_sum_irreducible}
&
The verification $c_{\min}(\Phi_d,E)>1$ for $d\in\{8,10,12,18,30\}$.
&
Algorithms \ref{alg:min-partition} and
\ref{alg:cmin-exceptional}.
\\
\addlinespace

Lemma \ref{lemma_RT_sum_reducible} and
Table \ref{table_list_c<=1}
&
Exceptional values of $c_{\min}(\Phi_d^\pm,E)$, including equality
cases and minimizing decompositions.
&
Algorithms \ref{alg:min-partition} and
\ref{alg:cmin-exceptional}.
\\
\addlinespace

Lemma \ref{lem_RT_sum_A_B_empty} and
Table \ref{table_list_S(X,0)<=1}
&
Values of $s_{\min}(f,E)$ for the exceptional cases,
including the irreducible cases.
&
Algorithm \ref{alg:smin-exceptional}.
\\
\addlinespace

Lemma \ref{lemma_S_AA_BB}
&
Verification of
$S(A_f\sqcup A_{f'},B_f\sqcup B_{f'})\geq1$
and its equality cases.
&
Algorithm \ref{alg:exceptional-pair}.
\\
\addlinespace

Lemma \ref{lemma_S_XX}
&
Verification of $S(X_f,X_{f'})\geq1$ and its degree-two equality
case.
&
Algorithm \ref{alg:exceptional-pair}.
\\
\addlinespace

Proposition \ref{prop_one_nonlinear_factor}
&
The remaining cases with exactly one nonlinear irreducible factor.
&
Algorithms \ref{alg:smin-exceptional} and
\ref{alg:low-age-types}.
\\
\addlinespace

Tables \ref{list_canonical_sing} and \ref{list_noncanonical_sing}
&
Types in the interior arising from unit eigenvalues and from
nonlinear factors.
&
Remark \ref{rem:unit-interior-enumeration} and
Algorithm \ref{alg:low-age-types}.
\\

\bottomrule
\end{longtable}
\normalsize

\bibliographystyle{alpha}
\bibliography{main}

\end{document}